\documentclass[11pt,a4paper,reqno]{amsart}
\usepackage[english]{babel}
\usepackage[applemac]{inputenc}
\usepackage[T1]{fontenc}
\usepackage{palatino}
\usepackage{verbatim}
\usepackage{amsmath}
\usepackage{amssymb}
\usepackage{amsthm}
\usepackage{soul}
\usepackage{amsfonts}
\usepackage{graphicx}
\usepackage{mathtools}
\usepackage{hyperref}
\usepackage{color}

\author{Katrin F\"assler, Tuomas Orponen, and S\'everine Rigot}
\title{Semmes surfaces and intrinsic Lipschitz graphs in the Heisenberg group}
\address{University of Jyv\"askyl\"a, Department of Mathematics and Statistics, P.O. Box 35 (MaD), FI-40014 University of Jyv\"askyl\"a, FINLAND}
\email{katrin.s.fassler@jyu.fi}
\address{University of Helsinki, Department of Mathematics and Statistics, P.O. Box 68 (Pietari Kalmin katu 5), FI-00014 University of Helsinki, FINLAND}
\email{tuomas.orponen@helsinki.fi}
\address{Universit\'e C\^ote d'Azur, CNRS, LJAD, Parc Valrose, 06108 Nice Cedex 02, FRANCE}
\email{Severine.RIGOT@univ-cotedazur.fr}


\newcommand{\R}{\mathbb{R}}

\newcommand{\LL}{\mathbb{L}}
\newcommand{\He}{\mathbb{H}}
\newcommand{\N}{\mathbb{N}}

\newcommand{\W}{\mathbb{W}}
\newcommand{\Z}{\mathbb{Z}}

\newcommand{\calL}{\mathcal{L}}
\newcommand{\calD}{\mathcal{D}}
\newcommand{\calG}{\mathcal{G}}
\newcommand{\calH}{\mathcal{H}}
\newcommand{\calB}{\mathcal{B}}

\newcommand{\diam}{\operatorname{diam}}
\newcommand{\card}{\operatorname{card}}
\newcommand{\dist}{\operatorname{dist}}

\newcommand{\NM}{\textup{NM}}
\newcommand{\NC}{\textup{NC}}
\newcommand{\h}{\mathfrak{h}}
\newcommand{\wid}{\mathrm{width}}

\newcommand{\Hek}{\mathbb{H}^{k}}
\newcommand{\Rk}{\mathbb{R}^{2k}}

\newcommand{\Sk}{\mathbb{S}^{2k-1}}
\newcommand{\calHk}{\mathcal{H}^{2k+1}}
\newcommand{\perh}{\operatorname{Per}_{\He}}
\newcommand{\per}{\operatorname{Per}}

\numberwithin{equation}{section}

\theoremstyle{plain}
\newtheorem{thm}[equation]{Theorem}

\newtheorem{lemma}[equation]{Lemma}

\newtheorem{cor}[equation]{Corollary}
\newtheorem{proposition}[equation]{Proposition}

\theoremstyle{definition}

\newtheorem{definition}[equation]{Definition}

\newtheorem{ex}[equation]{Example}

\theoremstyle{remark}
\newtheorem{remark}[equation]{Remark}

\begin{document}

\begin{abstract}
A Semmes surface in the Heisenberg group is a closed set $S$ that is upper Ahlfors-regular with codimension one and  satisfies the following condition, referred to as Condition B. Every ball $B(x,r)$ with $x \in S$ and $0 < r < \diam S$ contains two balls with radii comparable to $r$ which are contained in different connected components of the complement of $S$. Analogous sets in Euclidean spaces were introduced by Semmes in the late $80$'s. We prove that Semmes surfaces in the Heisenberg group are lower Ahlfors-regular with codimension one and have big pieces of intrinsic Lipschitz graphs. In particular, our result applies to the boundary of chord-arc domains and of reduced isoperimetric sets. The proof of the main result uses the concept of quantitative non-monotonicity developed by Cheeger, Kleiner, Naor, and Young. The approach also yields a new proof for the big pieces of Lipschitz graphs property of Semmes surfaces in Euclidean spaces.
\end{abstract}

\maketitle

\tableofcontents


\section*{Acknowledgements}

K.F.~is supported by Swiss National Science Foundation via the project \emph{Intrinsic rectifiability and mapping theory on the Heisenberg group}, grant no.~$161299$. T.O.~is supported by the Academy of Finland via the project \emph{Quantitative rectifiability in Euclidean and non-Euclidean spaces}, grant no.~$309365$, and by the University of Helsinki via the project \emph{Quantitative rectifiability of sets and measures in Euclidean spaces and Heisenberg groups}, grant no. $75160012$. S.R.~is partially supported by the French National Research Agency, \emph{Sub-Riemannian Geometry and Interactions} ANR-15-CE40-0018 project. Part of the research was conducted while K.F. and S.R. were visiting the University of Helsinki. The hospitality of the university is gratefully acknowledged. We would like to thank Bruno Franchi, Vasilis Chousionis, and Sean Li for numerous fruitful discussions, and Sylvester Eriksson-Bique for pointing out that the lower Ahlfors regularity of Semmes surfaces follows from 
relative isoperimetric inequalities. We are also grateful to the reviewer for pointing out relevant references.


\section{Introduction} \label{sect:introduction}

A study of quantitative rectifiability has recently been initialised in the Heisenberg groups $\Hek$, modelled after the Euclidean theory of \emph{uniformly rectifiable sets}, developed by David and Semmes~\cite{DS1,MR1251061} in the 90's. A starting point for the theory in $\Hek$ is the notion of \emph{intrinsic Lipschitz graphs} introduced by Franchi, Serapioni, and Serra Cassano~\cite{FSS,MR2313532}, see also~\cite{MR3511465}. Some evidence already suggests that intrinsic Lipschitz graphs are suitable building blocks for developing a theory {of quantitative rectifiability}, see for instance the papers of Naor and Young~\cite{Naor:2017:IGG:3055399.3055413,NY} and of the first two authors with Chousionis~\cite{CFO} (we note that the class of intrinsic Lipschitz graphs studied in~\cite{FSS, CFO} coincides with the class considered in~\cite{Naor:2017:IGG:3055399.3055413,NY}, see Remark~\ref{rmk:intrinsic-lip-graphs-H2k+1}). The present paper continues this line of research and provides further evidence in this direction.

We focus  on \emph{Semmes surfaces} in $\Hek$. These sets were introduced in Euclidean spaces by Semmes~\cite{Semmes} in connection with the study of the boundedness of certain singular integrals. David~\cite{MR1009120} proved that Semmes surfaces in $\R^{n}$ are uniformly rectifiable, and they even satisfy the stronger property of  having \emph{big pieces of Lipschitz graphs}. The main result in the present paper is the Heisenberg analogue of this theorem.

Before discussing the results more thoroughly, we recall some definitions. We refer to Section~\ref{s:prelim} for more details about the Heisenberg group $\Hek$, which agrees with $\Rk\times\R$ as a set and which can be viewed as a metric space of Hausdorff dimension $2k+2$ when equipped with the Kor\'anyi metric.


\begin{definition}[Upper and lower Ahlfors-regular sets] \label{def:Ahlfors-regular}
Let $k \in \N$ and $0 \leq s \leq 2k + 2$. A closed set $E\subset \Hek$ is \emph{upper Ahlfors-regular with dimension $s$} if there is a constant $C >0$ such that
\begin{equation} \label{e:upper-regular}
\mathcal{H}^{s} (B(p,r) \cap E) \leq C \, r^{s} ,  \qquad  p \in \Hek, \quad r>0 .
\end{equation}
Similarly, a closed set $E\subset \Hek$ is \emph{lower Ahlfors-regular with dimension $s$} if there is a constant $c >0$ such that
\begin{equation} \label{e:lower-regular}
\mathcal{H}^{s} (B(p,r) \cap E) \geq c \, r^{s}  ,  \qquad  p \in E, \quad 0 < r < \diam E.
\end{equation}
If $E$ is both upper and lower Ahlfors-regular with dimension $s$, then $E$ is called \emph{Ahlfors-regular with dimension $s$}.
\end{definition}

We will mostly be concerned with Ahlfors-regular sets with dimension $2k + 1$. So we will shorten the terminology and say that a closed subset of $\Hek$ is upper/lower Ahlfors-regular, without specifying the dimension, to mean that it is upper/lower Ahlfors-regular with dimension $2k + 1$.

We denote by $A^c:=\Hek\setminus A$ the complement of a set $A\subset \Hek$.

\begin{definition}[Condition \textbf{B}] \label{def:conditionB}
A closed set $S$ satisfies \emph{Condition} \textbf{B} if there is a constant $c > 0$ such that the following holds. For all $p \in S$ and $0 < r < \diam S$, there exist $q_{1},q_{2}$ in different connected components of $S^c$ such that
\begin{equation}\label{corkscrew}
q_{1},q_{2} \in B(p,r) \quad \text{and} \quad \min \{\dist(q_{1},S),\dist(q_{2},S)\} > c\, r.
\end{equation}
\end{definition}

Semmes surfaces are then defined as follows.

\begin{definition}[Semmes surfaces]\label{B} A closed set $S \subset \Hek$ is called a \emph{Semmes surface} if $S$ is upper Ahlfors-regular and satisfies Condition \textbf{B}.
\end{definition}

We refer to Definition~\ref{d:intrLip-H2k+1} for the definition of intrinsic $L$-Lipschitz graphs in $\Hek$ and give below the definition of a set with big pieces of intrinsic Lipschitz graphs.

\begin{definition}[BPiLG] \label{def:BPiLG} A closed set $S \subset \Hek$ has \emph{big pieces of intrinsic Lipschitz graphs} (or BPiLG in short) if there exist constants $L>0$ and $c > 0$ such that the following holds. For all $p \in S$ and $0 < r < \diam S$,  there exists an intrinsic $L$-Lipschitz graph $\Gamma$ such that
\begin{displaymath} \calHk (B(p,r) \cap S \cap \Gamma) \geq c\, r^{2k+1}. \end{displaymath}
\end{definition}

As in the Euclidean setting, BPiLG is a scale-invariant and quantitative notion of rectifiability.
It implies $(2k+1)$-dimensional $\He$-rectifiability, or equivalently $\He_L$-rectifiability, in the sense of~\cite[Definition 4.101]{MR3587666}. Namely, up to a set of null $\calHk$ measure, a closed Ahlfors-regular set with BPiLG can be covered by countably many intrinsic Lipschitz graphs, but the BPiLG condition is significantly stronger than $\He$-rectifiability.

Our main result reads as follows.

\begin{thm}\label{main}
Semmes surfaces in $\Hek$ are lower Ahlfors-regular and have big pieces of intrinsic Lipschitz graphs.
\end{thm}

The constants $L$ and $c$ in the BPiLG property can be chosen depending only on $k$, and on the upper Ahlfors-regularity and Condition~\textbf{B} constants for the Semmes surface. The lower Ahlfors-regularity constant only depends on $k$, and the Condition~\textbf{B} constant. The lower Ahlfors-regularity of closed sets satisfying Condition~\textbf{B} actually holds in much more general metric measure spaces than $\Hek$, see Remark \ref{r:lower_ADR}.

Before we give an outline of the proof of Theorem~\ref{main}, we give some applications. We first recall the corkscrew condition for open sets.

\begin{definition} [Corkscrew condition] \label{def:corkscrew-cdt} An open set $\Omega \subset \Hek$ satisfies the \emph{corkscrew condition} if for every $p \in \partial \Omega$ and $0 < r < \diam \partial \Omega$, there exist points $q_{1} \in \Omega$ and $q_{2} \in \overline{\Omega}^c$ satisfying~\eqref{corkscrew} with $S:=\partial \Omega$.
\end{definition}

It is clear that the boundary of an open set satisfying the corkscrew condition satisfies Condition~\textbf{B}. However, there are Semmes surfaces that do not arise as the boundary of some open set satisfying the corkscrew condition, see Section~\ref{subsect:NY-vs-our-setting} for more comments in this direction. Going back to the corkscrew condition, Theorem~\ref{main} has the following immediate consequence.

\begin{cor}\label{cor:corkscrew}
Let $\Omega \subset \Hek$ be an open set satisfying the corkscrew condition whose boundary is upper Ahlfors-regular. Then $\partial \Omega$ is lower Ahlfors-regular and has BPiLG.
\end{cor}

Corollary~\ref{cor:corkscrew} should be viewed as a quantitative counterpart of the  following qualitative statement. If $\Omega \subset \Hek$ is an open set satisfying the corkscrew condition whose boundary has locally finite $\calHk$ measure, then $\partial \Omega$ is $(2k+1)$-dimensional $\He$-rectifiable, or equivalently $\He_L$-rectifiable, in the sense of~\cite[Definition 4.101]{MR3587666}. This follows from the theory of finite $\He$-perimeter sets in $\Hek$, developed by Franchi, Serapioni, and Serra Cassano in~\cite{FSSC}, by noting that the topological boundary of such a set  coincides with its measure-theoretic boundary.

We recall that NTA domains are examples of domains satisfying the corkscrew condition together with an additional (interior) Harnack chain condition. In Euclidean spaces, NTA domains were introduced in 1982 by Jerison and Kenig~\cite{JK} and have found numerous applications, in particular in the study of harmonic functions and in the theory of elliptic equations and free boundary problems. The boundary of an arbitrary NTA domain need not be rectifiable, let alone have big pieces of Lipschitz graphs. Imposing the additional assumption that the boundary is Ahlfors-regular leads to the notion of a chord-arc domain, as stated in Definition~\ref{def:chord-arc} below in the Heisenberg setting. For more information and references, we refer for instance to the recent papers \cite{MR3626548} of Azzam, Hofmann, Martell, Nystr\"om, and Toro, \cite{MR3738187} of Azzam, and \cite{Badger} of Badger in the Euclidean setting. NTA domains in $\Hek$ have been investigated by Capogna and Tang \cite{MR1323792}, and in more general Carnot groups by Capogna and Garofalo \cite{CG} and Monti and Morbidelli \cite{MM}. See also the survey \cite{MR1847513} by Capogna, Garofalo and Nhieu.

\begin{definition} [Chord-arc domains] \label{def:chord-arc}
An NTA domain with Ahlfors-regular boundary in $\Hek$ is called a \emph{chord-arc domain}.
\end{definition}

Corollary~\ref{cor:corkscrew} applies in particular to chord-arc domains.

\begin{cor}\label{cor:chord-arc}
The boundary of a chord-arc domain in $\Hek$ has BPiLG.
\end{cor}

To illustrate Corollary~\ref{cor:chord-arc}, we next discuss a simple example of a chord-arc domain in $\Hek$ that is not an intrinsic Lipschitz domain, even though it has smooth boundary in the Euclidean sense. The same example was presented in~\cite[Proposition 15]{CG} as a motivation for studying NTA domains in $\He^1$.

\begin{ex} \label{ex:H} Let $\Omega := \{(v,t) \in \Hek : t > 0\}$ be the upper half-space whose boundary is $H: = \Rk \times \{0\}$. It is easy to see that $H \cap B(0,r)$ is not an intrinsic Lipschitz graph for any $r > 0$. Nevertheless, we now check that $\Omega$ is a corkscrew domain with upper Ahlfors-regular boundary, so $H$ has BPiLG by Corollary~\ref{cor:corkscrew}. First, $\Omega$ is actually an NTA domain. This follows for instance from a result of Monti and Morbidelli~\cite{MM} since $H$ is smooth. Alternatively, one can check that $\Omega$ satisfies the corkscrew condition via direct geometric arguments. Next, the upper Ahlfors-regularity of $H$ follows directly from \cite[Theorem 1.2]{MR2250055}, or can be deduced from~\cite[Corollary~7.7(i)]{FSSC}. Indeed, using left translations and dilations, the upper Ahlfors-regularity of $H$ is equivalent to
\begin{equation*}
\sup_{p\in H} \calHk ( B(0,1) \cap p\cdot H) <+\infty.
\end{equation*}
If $P:=p\cdot H$ for some $p\in H$, then $P$ is the boundary of a $C^1$ domain, and it follows from~\cite[Corollary~7.7(i)]{FSSC} that
\begin{equation*}
\calHk ( B(0,1)\cap P) \sim_k \int_{ B(0,1)\cap P} |C(q) n| \, d\mathcal{H}^{2k}_\mathrm{Euc}(q)
\end{equation*}
where $n$ denotes the unit Euclidean normal to $P$, the expression of $|C(p) n|$ is given in~\eqref{e:C(p)n}, and $\mathcal{H}^{2k}_\mathrm{Euc}$ denotes the $2k$-dimensional Hausdorff measure with respect to the Euclidean distance. Since $B(0,1)$ is contained in a Euclidean ball with radius $2$, it follows that $\calHk (B(0,1)\cap P) \lesssim_k 1$ which proves the upper Ahlfors-regularity of $H$. The lower Ahlfors-regularity of $H$ will be proven in \eqref{e:lower-ar-hyperplanes}, or follows again from \cite[Theorem 1.2]{MR2250055}. This last remark together with the previous ones hence show that $\Omega$ is a chord-arc domain.
\end{ex}

Other examples of Semmes surfaces in $\Hek$ are given by the boundaries of reduced isoperimetric sets. We recall that a measurable set $E\subset \Hek$ is said to be \emph{isoperimetric} if
$$\perh(E,\Hek) \leq \perh(F,\Hek)$$
for every measurable set $F\subset \Hek$ such that $\calH^{2k+2} (F) = \calH^{2k+2} (E)$, where $\perh$ denotes the $\He$-perimeter. We refer to~\cite{FSSC} for the definition of $\He$-perimeter. The existence of isoperimetric sets in $\Hek$, and more generally in any Carnot group, has been proven by the third author with Leonardi \cite[Theorem~3.2]{LR}. It is still an open question to find explicitly which sets are isoperimetric in this setting. One of the reasons why this question remains open is the difficulty of proving a priori regularity estimates for isoperimetric sets. It was however shown in~\cite{LR} that every isoperimetric set $E$ is equivalent to a unique open set $\Omega$, meaning that $\calH^{2k+2} (E \bigtriangleup\Omega) = 0$, which satisfies the corkscrew condition and whose boundary is Ahlfors-regular \cite[Theorem~3.3]{LR}. Such a set $\Omega$ is still isoperimetric and is called a \emph{reduced isoperimetric} set. Corollary~\ref{cor:corkscrew} can hence be applied to boundaries of reduced isoperimetric sets.

\begin{cor} \label{cor:isoperimetric-sets}
The boundary of a reduced isoperimetric set in $\Hek$ has BPiLG.
\end{cor}

We give now  an outline of the proof of Theorem \ref{main}, the details of which can be found in Sections~\ref{sect:main-approximation} to~\ref{sect:WGL+BVP-BPiLG-H2k+1}. As already mentioned, Theorem~\ref{main} is the Heisenberg analogue of David's Euclidean result~\cite[Proposition~2]{MR1009120}. Several different proofs for the Euclidean result, besides~\cite{MR1009120}, are available in the literature and can be found in~\cite{DJ} and~\cite{MR1132876}. The general strategy of our proof is inspired by the strategy in~\cite{MR1132876}. However, as far as we can tell, the arguments to make this strategy work in the Euclidean setting do not easily translate to $\Hek$, and neither do the other known proofs of the Euclidean analogue of Theorem~\ref{main}. Our arguments to implement the strategy in the Heisenberg setting will hence be quite different from Euclidean ones, and are inspired by the concept of \emph{non-monotonicity} introduced by Cheeger, Kleiner and Naor in~\cite{CKN}, and by some of the techniques developed by Naor and Young in~\cite{NY}.

Following the general strategy of~\cite{MR1132876}, we prove that a Semmes surface $S$ is lower Ahlfors-regular and has BPiLG by showing that $S$ has the following properties. First, $S$ is lower Ahlfors-regular by Proposition~\ref{prop:lowerAR}. This will be obtained as a consequence of the relative isoperimetric inequality in $\Hek$. Next, $S$ has \emph{big vertical projections} (BVP), see Definition~\ref{def:BVP} and Proposition~\ref{BVPProp}. Roughly speaking, BVP is a uniform and scale invariant way of requiring that, for every ball $B$ centred on $S$, there is a vertical subgroup onto which $B \cap S$ has a large projection in $\calHk$ measure. Finally, $S$ satisfies the \emph{weak geometric lemma} (WGL) \emph{for vertical hyperplanes}, see Definition~\ref{def:wgl-vertical-hyperplanes} and Proposition~\ref{prop:wgl-vertical-hyperplanes}. Informally, WGL for vertical hyperplanes states that $S$ can be well-approximated by vertical hyperplanes at most points and scales. For closed Ahlfors-regular sets in $\Hek$, the implication
\begin{equation*}
\textrm{BVP} + \textrm{WGL for vertical hyperplanes}   \quad \Longrightarrow \quad \textrm{BPiLG} \end{equation*}
is the Heisenberg analogue of a Euclidean result due to David and Semmes~\cite[Theorem~1.14]{MR1132876}. It has been proven in $\He^1$
by the first two authors with Chousionis in~\cite{CFO}. The extension of this result to higher dimensional Heisenberg groups will be given in Section~\ref{sect:WGL+BVP-BPiLG-H2k+1}, see~Theorem~\ref{thm:BVP+WGL-BPiLG}, and this will allow us to conclude the proof of Theorem~\ref{main}.

A challenge in $\Hek$ is to deduce from Condition~\textbf{B}, which does not include
any reference to the horizontal distribution of the Heisenberg group, information
about the family of vertical projections and their fibers, which
are horizontal lines. This challenge is not present in $\R^n$, where lines in all directions appear as fibers of orthogonal projections.
Moreover, unlike orthogonal projections in $\R^{n}$, group projections onto a vertical subgroup are not Lipschitz maps when both $\Hek$ and the vertical subgroup are equipped with the Kor\'anyi distance.
Consequently, while showing that a Semmes surface in $\R^n$ has big projections is a rather immediate consequence of Condition~\textbf{B}, proving the corresponding results in $\Hek$ is not so simple. Complications related to the fact that we are dealing with a restricted family of lines also appear in the proof of WGL for vertical hyperplanes.

We introduce in Section~\ref{subsect:width} the notion of \emph{horizontal width} as a central tool in the proof of both BVP and WGL for vertical hyperplanes. This notion is inspired by the concept of non-monotonicity, introduced by Cheeger, Kleiner, and Naor in~\cite{CKN}. The non-monotonicity of a given set is indeed controlled by the horizontal width of its boundary, see Lemma~\ref{lemma1} and Proposition~\ref{form1}. The novelty here is that horizontal width turns out to be a suitable tool to deal with Semmes surfaces that do not necessarily arise as boundaries of sets satisfying the corkscrew condition.

We show that if the horizontal width of a Semmes surface $S$ is small inside a ball, then $S$ is well-approximated, in a bilateral way, by a hyperplane inside a slightly smaller ball, Corollary~\ref{cor:main}. Our proof goes through a first step, Corollary~\ref{cor:small-width-imply-cc-close-half-spaces}, which combines the control of the non-monotonicity of every connected components of the complement of $S$ by the horizontal width of $S$ with a result due to Naor and Young~\cite[Proposition~66]{NY}. Naor and Young's result is itself inspired by a deep stability  theorem for monotone sets due to Cheeger, Kleiner and Naor. Condition~\textbf{B} first appears in the second step of the proof of Corollary~\ref{cor:main} and allows us to prove that if $S$ satisfies the conclusion of Corollary~\ref{cor:small-width-imply-cc-close-half-spaces}, namely, if all connected components of its complement are measure-theoretically close to half-spaces inside a ball, then $S$ is rather flat in a slightly smaller ball, see Proposition~\ref{monotonicityToBetas}.

With Corollary~\ref{cor:main} in hand, we can then proceed to the proof of BVP and WGL for vertical hyperplanes. The proof of BVP is given in Section~\ref{BVPSection}. Then Section~\ref{sect:bwgl} is devoted to the proof of the validity of the \textit{bilateral} weak geometric lemma (BWGL) for  vertical hyperplanes, see Definition~\ref{def:bwgl-vertical-hyperplanes} and Proposition~\ref{prop:bwgl-vertical-hyperplanes}. BWGL is stronger than WGL and obviously implies WGL for vertical hyperplanes, see Section~\ref{sect:BPiLG}. Our proof of the validity of BWGL for  vertical hyperplanes relies on the following facts. First, we prove that, for a closed upper Ahlfors-regular set, balls with large horizontal width satisfy a Carleson packing condition, Proposition~\ref{NMCarleson}. Combined with Corollary~\ref{cor:main}, this implies that a Semmes surface is well-approximated, in a bilateral way, by arbitrary hyperplanes at most points and scales, Proposition~\ref{prop:bwgl-arbitrary-hyperplanes}. Then we upgrade this approximation by arbitrary hyperplanes to a similar approximation by vertical hyperplanes, Proposition~\ref{prop:bwgl-vertical-hyperplanes}. We explicitly note in Section~\ref{sect:bwgl-vertical-hyperplanes} that the validity of BWGL for vertical hyperplanes for a closed Ahlfors-regular set is equivalent to the validity of BWGL for arbitrary hyperplanes, Theorem~\ref{thm:bwgl}. The proof of this latter fact is already essentially contained in the work of Naor and Young~\cite{NY}.

We end the paper with a
section where we first
comment on connections as well as differences, between Naor and Young's work \cite[Section~9]{NY} and the setting and results of the present paper. This suggests
further questions to better understand the theory of ``uniform rectifiability'' that is now emerging in the Heisenberg setting. In the final Section~\ref{subsect:euclidean-setting}, we go back to Euclidean spaces and note that the method of the current paper applies to the Euclidean setting, too, and gives a direct proof of the validity of the bilateral weak geometric lemma for Semmes surfaces in $\R^n$. As a consequence, it also gives a new path for the proof of the big pieces of Lipschitz graphs property for Semmes surfaces in Euclidean spaces.



\section{Preliminaries} \label{sect:preliminaries}

In this section, we introduce some of the concepts frequently used in the rest of the paper and state notational conventions.


\subsection{The Heisenberg group} \label{s:prelim}

For a fixed integer $k \geq 1$, we identify the Heisenberg group $\Hek$ with $\Rk \times \R$ equipped with the group law
\begin{displaymath}
(v,t) \cdot (v',t') :=\left(v+v',t+t'+\omega(v,v')/2\right),\quad (v,t),(v',t')\in \Hek,
\end{displaymath}
where $\omega(v,v') := \sum_{j=1}^k v_j v'_{j+k} - v_{j+k} v'_j$ for $v=(v_1,\dots,v_{2k})$, $v'=(v'_1,\dots,v'_{2k}) \in \Rk$.

We equip $\Hek$ with the \emph{Kor\'{a}nyi norm} $\|\cdot\|$ and \emph{Kor\'{a}nyi metric} $d$ defined respectively by
 \begin{equation*}
 \|(v,t)\|:= \sqrt[4]{|v|^4 + 16 t^2} \quad \text{and}\quad d(p,q):= \| q^{-1}\cdot p\| ,
 \end{equation*}
where $|\cdot|$ denotes the Euclidean norm in $\Rk$.

The metric $d$ is homogeneous, meaning that it is left-invariant, that is, $d(p\cdot q, p\cdot q') = d(q,q')$ for all $p, q, q' \in \Hek$, and one-homogeneous with respect to the \emph{Heisenberg dilations} $(\delta_r)_{r>0}$ given by
 \begin{displaymath}
\delta_r:\Hek \to \Hek,\quad  \delta_r(v,t):= (rv,r^2t),
 \end{displaymath}
that is, $d(\delta_r(p), \delta_r(q)) = r d(p,q)$ for all $p,q\in\Hek$ and $r>0$. We recall that the family $(\delta_r)_{r>0}$ of dilations is a one-parameter
group
 of group automorphisms of $\Hek$. We also recall that any two homogeneous distances on $\Hek$ are biLipschitz equivalent.

For $s \geq 0$, we denote by $\calH^s$ the Hausdorff measure of dimension $s$ with respect to the Kor\'anyi metric. We recall that $\calH^{2k+2}$ is a Haar measure on $\Hek$ and is $(2k+2)$-uniform. In particular $(\Hek,d)$ is a metric space of Hausdorff dimension $2k+2$. A thorough introduction to $(\Hek,d)$ can be found for instance to the monograph~\cite{MR2312336}.

Unless otherwise explicitly stated, all metric notions in $\Hek$ will always be defined with respect the Kor\'anyi metric. In particular, we write $B(p,r):=\{q \in \Hek : d(p,q) \leq r\}$ for a closed ball in $(\Hek,d)$ centred at $p \in \Hek$ with radius $r > 0$. By default, a ball will refer to a closed ball.


\subsection{Subgroups, projections, and intrinsic Lipschitz graphs} \label{subsect:subgroups-etc-H2k+1}
We denote by $\Sk$ the Euclidean unit sphere in $\Rk$. Given $A\subset\Rk$, we denote by $A^\perp$ the linear subspace orthogonal to $A$ in the Euclidean sense in $\Rk$.

Given $\nu \in \Sk$, we set
$$ \LL_\nu:=\R\nu \times \{0\} \quad \text{and} \quad \W_\nu:= \nu^\perp \times \R.$$
These sets are complementary homogeneous subgroups of $\Hek$ which means that both $\LL_\nu$ and $\W_\nu$ are subgroups of $\Hek$ closed under the family of dilations $(\delta_r)_{r>0}$ and every point $p\in\Hek$ can be uniquely written as $p=p_{\W_\nu} \cdot p_{\LL_\nu}$ for some $p_{\W_\nu}\in \W_\nu$ and $p_{\LL\nu}\in \LL_\nu$. We call \textit{horizontal subgroup} a set of the form $\LL_\nu$ for some $\nu\in\Sk$ and \textit{vertical subgroup} a set of the form $\W_\nu$ for some $\nu\in\Sk$.

Given $\nu \in \Sk$, we define the \textit{vertical} and \textit{horizontal projections} onto $\W_\nu$ and $\LL_\nu$ respectively by
\begin{align*}
\pi_{\W_\nu}:\Hek \to \W_\nu,\quad &\pi_{\W_\nu}(p)=p_{\W_\nu}\\
\pi_{\LL_\nu}:\Hek \to \LL_\nu,\quad &\pi_{\LL_\nu}(p)=p_{\LL_\nu}.
\end{align*}
Note that, given $\nu\in \Sk$ and $p=(v,t)\in\Hek$, we have
\begin{align}
\pi_{\W_\nu}(p) &= \left(v-\langle v,\nu \rangle \nu, t- \omega(v,\langle v,\nu \rangle \nu)/2\right), \label{e:pwnu}\\
\pi_{\LL_\nu}(p) &= (\langle v,\nu \rangle \nu,0),\label{e:plnu}
\end{align}
where $\langle \cdot ,\cdot \rangle$ denotes the Euclidean scalar product in $\Rk$. Unlike orthogonal projections in Euclidean spaces, vertical projections $\pi_{\W_\nu}$ are not Lipschitz maps from $(\Hek,d)$ to $\W_\nu$ equipped with the restriction of the Kor\'anyi distance.

Given $\gamma > 0$ and $\nu\in\Sk$, we define the \emph{Heisenberg cone} with vertex at origin, opening $\gamma$ and direction $\nu$ by
\begin{displaymath} C_{\gamma}(\nu) := \{p \in \Hek: \, \|\pi_{\W_\nu}(p)\| \leq \gamma \|\pi_{\LL_\nu}(p)\|\}. \end{displaymath}

\begin{definition}[Intrinsic Lipschitz graphs] \label{d:intrLip-H2k+1} Let $L>0$ and $\nu\in\Sk$. We say that a set $\Gamma \subset \Hek$ is an \emph{intrinsic $(L,\nu)$-Lipschitz graph} (over $\W_{\nu}$)  if $\pi_{\W_\nu} (\Gamma) = \W_\nu$ and
\begin{equation} \label{e:def-Lip-graphs}
(p\cdot C_{1/L}(\nu))\cap \Gamma =\{p\} \quad \text{for all } p\in \Gamma.
\end{equation}
We say that $\Gamma \subset \Hek$ is an \emph{intrinsic $L$-Lipschitz graph} if it is an \emph{intrinsic $(L,\nu)$-Lipschitz graph} for some $\nu\in\Sk$, and we say  that $\Gamma \subset \Hek $ is an \emph{intrinsic Lipschitz graph} if it is a \emph{intrinsic $(L,\nu)$-Lipschitz graph} for some $L>0$ and $\nu\in\Sk$.
\end{definition}

This definition was introduced by Franchi, Serapioni and Serra Cassano, see for instance~\cite{FSS,MR2313532}. An intrinsic $(L,\nu)$-Lipschitz graph $\Gamma$ as in Definition~\ref{d:intrLip-H2k+1} can be written as $\Gamma = \{p\cdot\varphi(p): p\in \W_\nu\}$ for some map $\varphi: \W_\nu \rightarrow \LL_\nu$. Moreover, if $\Gamma \subset \Hek$ satisfies~\eqref{e:def-Lip-graphs} for some $L>0$ and $\nu\in\Sk$, then there is $L'>0$, depending only on $L$, and an intrinsic $(L',\nu)$-Lipschitz graph $\Gamma'$ (over $\W_{\nu}$) as in Definition~\ref{d:intrLip-H2k+1} such that $\Gamma \subset \Gamma'$. A detailed discussion of intrinsic Lipschitz graphs and their properties can be found for instance in Serra Cassano's lecture notes~\cite{MR3587666}.

The class of intrinsic Lipschitz graphs given by Definition~\ref{d:intrLip-H2k+1} is invariant under a change of homogeneous norm on $\Hek$. Namely, since any two homogeneous norms on $\Hek$ are bilipschitz equivalent, replacing the Kor\'anyi norm by any other homogeneous norm in the definition of the cones $C_{\gamma}(\nu)$ gives rise to the same class of intrinsic Lipschitz graphs.

\begin{remark} \label{rmk:intrinsic-lip-graphs-H2k+1}
A different definition for intrinsic Lipschitz graphs was given by Naor and Young in \cite[Section 2.3]{NY} using the cones
\begin{displaymath}
\widetilde{C}_{\lambda}(\nu) := \{p\in \Hek:\;\lambda d_{SR}(p,0)< d_{SR}(\pi_{\LL_\nu}(p),0)\},
\end{displaymath}
where $\nu\in\Sk$ and $d_{SR}$ denotes the sub-Riemannian distance.
According to this definition, a set $\Gamma \subset \Hek$ is an intrinsic Lipschitz graph if there exist $\lambda \in (0,1)$ and $\nu \in \Sk$ such that
\begin{equation}\label{eq:NY-H2k+1}
(p\cdot \widetilde{C}_{\lambda}(\nu)) \cap \Gamma = \emptyset \quad \text{for all }  p\in \Gamma .
\end{equation}
It is easy to see that intrinsic Lipschitz graphs in the sense of \eqref{eq:NY-H2k+1} are intrinsic Lipschitz graphs in the sense of Definition \ref{d:intrLip-H2k+1}. Conversely, an intrinsic Lipschitz graph in the sense of Definition~\ref{d:intrLip-H2k+1} is an intrinsic Lipschitz graph in the sense of~\eqref{eq:NY-H2k+1} as we explain now. We note that, although the class of intrinsic Lipschitz graphs given by Definition~\ref{d:intrLip-H2k+1} is independent of the choice of the homogeneous norm, the Kor\'anyi norm will nevertheless play a role for the definition of the cones $C_{\gamma}(\nu)$ in the following argument. This argument works with some other choices of homogeneous norms, but does not work for every arbitrary choice.
 The Kor\'anyi norm has indeed the useful property that there exists $\epsilon >0$ such that for every $\nu\in\Sk$, one has
\begin{equation}\label{eq:useful-H2k+1}
\|p\|^4 \geq \|\pi_{\LL_\nu}(p)\|^4 + \epsilon \|\pi_{\W_\nu}(p)\|^4 \quad \text{for all }p\in \Hek.
\end{equation}
This sort of strict convexity property is not true for arbitrary homogeneous norms. For instance $N(v,t):= \max\{|v|,\sqrt{|t|}\}$ defines a homogeneous norm on $\Hek$ for which there exists a point $p$ with $\pi_{\W_\nu}(p)\neq 0$, yet $N(p)= N(\pi_{\LL_\nu}(p))$. Inequality \eqref{eq:useful-H2k+1} for the Kor\'anyi norm implies for every $\gamma>0$ the existence of $\lambda \in (0,1)$ such that
\begin{displaymath}
\widetilde{C}_{\lambda}(\nu) \subset C_{\gamma}(\nu) \setminus \{0\}
\end{displaymath} as can be seen by the following simple computation.
For $p\notin C_{\gamma}(\nu)$, one finds
\begin{align*}
\|p\|^4 &\geq  \|\pi_{\LL_\nu}(p)\|^4 + \epsilon \|\pi_{\W_\nu}(p)\|^4 \geq \|\pi_{\LL_\nu}(p)\|^4 + \epsilon \gamma^4 \|\pi_{\LL_\nu}(p)\|^4,
\end{align*}
so that
\begin{displaymath}
d_{SR}(p,0)^4\geq\|p\|^4 \geq \left(1+ \epsilon \gamma^4\right)\|\pi_{\LL_\nu}(p)\|^4 = \left(1+ \epsilon \gamma^4 \right)d_{SR}(\pi_{\LL_\nu}(p),0)^4,
\end{displaymath}
and hence $p\notin \widetilde{C}_{\lambda}(\nu)$ for $\lambda = (1+\epsilon \gamma^4 )^{-1/4}$.
\end{remark}


\subsection{Horizontal lines}

\begin{definition}[Horizontal lines] A \emph{horizontal line} in $\Hek$ is a left-translate of a horizontal subgroup, that is, a set of the form $p \cdot \LL_\nu$ for some $p\in\Hek$ and $\nu\in\Sk$. We denote by $\calL$ the family of all horizontal lines in $\Hek$.
\end{definition}

Considering horizontal lines in the Heisenberg setting is natural in light of the specific geometry of $\Hek$. Restricted to a horizontal subgroup, the Kor\'anyi metric agrees with the Euclidean distance. Using left-translations, it follows that horizontal lines endowed with the restriction of the Kor\'anyi metric are isometric to the real line with its standard Euclidean distance and consequently have Hausdorff dimension $1$ in $(\Hek,d)$. Moreover, one finds for every horizontal line $\ell \in \calL$ and for every segment $I \subset \ell$ that $\calH^1(I) = \mathrm{diam}(I)$. On the contrary, non-horizontal lines have Hausdorff dimension $2$ with respect to the Kor\'anyi metric.

Up to a multiplicative constant, there exists a unique non-trivial left-invariant measure $\h$ on $\calL$, in $\He^1$ see ~\cite[Section 4.2]{CKN}. The following expression for $\h$ will be convenient for our purposes. For a Borel set $\mathcal{A} \subset \calL$, we write
\begin{equation} \label{e:def-meas-h-H2k+1}
\h(\mathcal{A}) := \int_{\Sk} \calHk (\{p \in \W_{\nu} : p\cdot \LL_\nu \in \mathcal{A}\}) \, d\nu
\end{equation}
where $d\nu$ denotes the surface measure on $\Sk$.

Measures of this form appear in the kinematic formula given for Carnot groups in \cite[Proposition 3.13]{MR2165404}, but the presentation simplifies in our setting. It can indeed easily be checked that~\eqref{e:def-meas-h-H2k+1} defines a non-trivial measure $\h$ on $\calL$. Moreover, $\h$ is left-invariant, as we verify in Lemma~\ref{leftInvariance-H2k+1}. By uniqueness, $\h$ coincides, up to a multiplicative constant, with the left-invariant measures on $\calL$ used in~\cite{CKN} and~\cite{NY}. For a related discussion, see~\cite[Section 11.1]{CKN}. Note that $\h$ has the following homogeneity property,
$$\h(\delta_r(\mathcal{A})) = r^{2k+1} \h(\mathcal{A}), \quad \text{for every Borel set }\mathcal{A}\subset \calL , \, r>0.$$
As an immediate consequence of~\eqref{e:def-meas-h-H2k+1}, we also note for future reference that
\begin{equation} \label{e:h-l(A)}
\h(\calL (A)) = \int_{\Sk} \calHk (\pi_{\W_{\nu}}(A)) \, d\nu
\end{equation}
where $\calL(A) := \{\ell \in \calL : \ell \cap A \neq \emptyset\}$ denotes the set of horizontal lines meeting $A\subset\Hek$.

\begin{lemma}\label{leftInvariance-H2k+1} The measure $\h$ is left-invariant. Namely, for every Borel set $\mathcal{A} \subset \calL$ and $p \in \Hek$, one has $\h(p\cdot \mathcal{A}) = \h(\mathcal{A})$.
\end{lemma}

\begin{proof} Start by writing
\begin{equation}\label{form23-H2k+1} \h(\mathcal{A}) = \int_{\Sk} \calHk (\{p \in \W_{\nu} : p\cdot \LL_\nu \in \mathcal{A}\}) \, d\nu  = \int_{\Sk} \calHk (A_{\nu}) \, d\nu, \end{equation}
where $A_{\nu}: =\{p \in \W_{\nu} : p\cdot \LL_\nu \in \mathcal{A}\} $. Now, for $\nu \in \Sk$ and $p \in \Hek$ fixed, we define the map $\Phi_{p,\nu} \colon \W_{\nu} \to \W_{\nu}$ by
\begin{displaymath}
\Phi_{p,\nu}(q) := \pi_{\W_{\nu}}(p \cdot q) = p \cdot q \cdot \pi_{\LL_{\nu}}(p)^{-1}, \qquad q \in \W_{\nu}. \end{displaymath}
Note that
\begin{equation}\label{form24-H2k+1} \Phi_{p^{-1},\nu} = (\Phi_{p,\nu})^{-1}. \end{equation}
Then, we recall a result of Franchi and Serapioni \cite[Lemma 2.20]{MR3511465}, which states that $\Phi_{p,\nu}$, seen as a map of $\Rk$, has Jacobian identically equal to $1$, which implies that
\begin{displaymath}
\calHk(\Phi_{p,\nu}(A_{\nu})) = \calHk (A_{\nu}), \qquad p \in \Hek, \: \nu \in \Sk. \end{displaymath}
Then, we write
\begin{align*} \Phi_{p,\nu}(A_{\nu}) & = \{\Phi_{p,\nu}(q) \in \W_{\nu} : q \cdot \LL_{\nu} \in \mathcal{A}\}\\
& = \{q \in \W_{\nu} : \Phi_{p^{-1},\nu}(q) \cdot \LL_{\nu} \in \mathcal{A}\}\\
& = \{q \in \W_{\nu} : \pi_{\W_{\nu}}(p^{-1} \cdot q) \cdot \LL_{\nu} \in \mathcal{A}\}\\
& = \{q \in \W_{\nu} : q \cdot \LL_{\nu} \in p \cdot \mathcal{A}\},
\end{align*}
using \eqref{form24-H2k+1} in the passage to the second line. Consequently, by \eqref{form23-H2k+1},
\begin{equation*}
\begin{split}
\h(\mathcal{A}) &= \int_{\Sk} \calHk (\Phi_{p,\nu}(A_{\nu})) \, d\nu \\
& = \int_{\Sk} \calHk (\{q \in \W_{\nu} : q \cdot \LL_{\nu} \in p \cdot \mathcal{A}\}) \, d\nu = \h(p \cdot \mathcal{A}),
\end{split}
\end{equation*}
as claimed.
\end{proof}

\subsection{Notations} For $A,B > 0$, we write $A \lesssim_k B$ to mean that there is a constant $C >0$ whose value depends only on the dimensional parameter $k$ of $\Hek$ such that $A \leq CB$. We write $A \lesssim_{reg} B$ to mean that the value of $C$ is allowed to depend also on the regularity, that is, upper/lower Ahlfors-regularity and/or Condition~\textbf{B}, constants of the set or measure under consideration, in addition to the dimensional parameter $k$. Given an auxiliary parameter $h$, distinct from the dimensional parameter $k$, we write $A \lesssim_{h} B$ to mean that the value of $C$ is allowed to depend also on $h$, in addition to the data mentioned above.   We abbreviate the two-sided inequalities $A \lesssim_k B \lesssim_k A$, $A \lesssim_{reg} B \lesssim_{reg} A$, and $A \lesssim_{h} B \lesssim_{h} A$ by $A \sim_k B$, $A \sim_{reg} B$ and $A \sim_h B$ respectively. We shall also often shorten the terminology, saying that a constant depends only on the upper/lower Ahlfors-regularity and/or Condition~\textbf{B} constants of a given set, and/or on other auxiliary parameters, when the constant may also depend on the dimensional parameter $k$.


\section{Approximating Semmes surfaces with small width by hyperplanes}\label{sect:main-approximation}

In this section, we introduce the notion of \textit{horizontal width} and prove a bilateral  approximation by hyperplanes for Semmes surfaces with small width. See Corollary~\ref{cor:main} that will be a key tool in the proof of both BVP and BWGL for Semmes surfaces.


\subsection{Horizontal width} \label{subsect:width}

Given a horizontal line $\ell \in \calL$, we define the \emph{horizontal width} of a set $E\subset \Hek$ with respect to $\ell$ in a ball $B$ by
\begin{displaymath} \wid_{B}(E,\ell) := \diam(B \cap E \cap \ell). \end{displaymath}

Recall that balls in this paper are closed, unless otherwise specified. The definition of width is inspired by the notions of non-convexity and non-monotonicity introduced by Cheeger, Kleiner and Naor in \cite[Section 4.2]{CKN}. Given $\ell \in \calL$, a set $A \subset \Hek$ such that $A \cap \ell$ is $\calH^{1}$ measurable, and a ball $B \subset \Hek$, the \emph{non-convexity} $\NC_{B}(A,\ell)$ of $A$ with respect to $\ell$ on $B$ is defined by
\begin{displaymath} \NC_{B}(A,\ell) := \inf \left\{ \int_{B \cap \ell} |\chi_{A} - \chi_{I}| \, d\calH^{1} : I \subset \ell \text{ is an interval }\right\} \end{displaymath}
where we allow the interval $I$ to be empty in the infimum above. The \emph{non-monotonicity} $\NM_{B}(A,\ell)$ of $A$ with respect to $\ell$ on $B$ is then defined by
\begin{equation*} \label{e:def-non-monotonicity}
 \NM_{B}(A,\ell) := \NC_{B}(A,\ell) + \NC_{B}(A^c,\ell).
\end{equation*}

For a fixed line $\ell \in \calL$, the relationship between the non-monotonicity of a set and the horizontal width of its boundary is given by Lemma \ref{lemma1} below. We start with the following observation regarding non-convexity. Recall that $\calL(A) := \{\ell \in \calL : \ell \cap A \neq \emptyset\}$ denotes the set of horizontal lines meeting $A\subset\Hek$.

\begin{lemma}\label{l:convex}
Let $A \subset \Hek$, let $B$ be a ball, and let $\ell \in  \calL(B\cap \partial A)$ be a line such that $A \cap \ell$ is $\calH^{1}$ measurable. Then,
\begin{displaymath} \NC_{B}(A^c,\ell) \leq \calH^1(A \cap J), \end{displaymath}
where, picking any orientation for $\ell$, we denote by $J$ the segment from $\min (B \cap \partial A \cap \ell)$ to $\max (B \cap \partial A \cap \ell)$.
\end{lemma}

\begin{proof}
Note that since $B$ is Euclidean convex when identifying $\Hek$ with $\Rk\times \R$, the set $B\cap \ell$ is an interval. We may assume that $B \cap A^{c} \cap \ell \neq \emptyset$, otherwise $\NC_{B}(A^c,\ell)$ is trivially zero. Let $I$ be the (possibly degenerate) segment from $\inf (B \cap \ell \cap A^{c})$ to $\sup (B \cap \ell \cap A^{c})$.
 We observe that
\begin{equation}\label{form10} |\chi_{A^{c}}(p) - \chi_{I}(p)| \leq \chi_{A \cap I}(p), \qquad p \in B \cap \ell. \end{equation}
Indeed, if the expression on the left-hand side is non-zero, then either $p \in A^{c} \setminus I$ or $p \in I \setminus A^{c}$. But the first case cannot occur, since $B \cap A^{c} \cap \ell \subset I$ by definition of $I$. In the second case, $p \in A \cap I$ as claimed.

\begin{figure}[h!]
\begin{center}
\includegraphics[scale = 0.5]{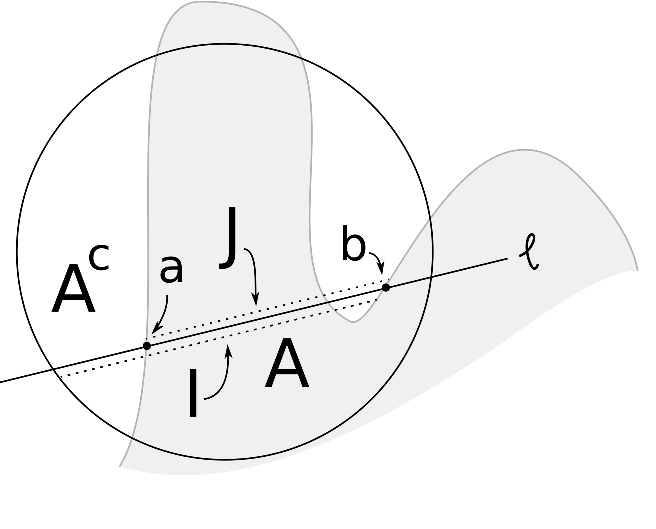}
\caption{The set $A$, and segments $I$ and $J$.}\label{fig1}
\end{center}
\end{figure}

Next, we note that
\begin{equation}\label{form11} A \cap I \subset A \cap J. \end{equation}
To see this, it suffices to argue that $I \setminus J \subset A^{c}$. We write $a:=\min (B \cap \partial A \cap \ell)$ and $b:=\max (B \cap \partial A \cap \ell)$ so that $J=[a,b]$, see Figure \ref{fig1}. Note that $I \setminus J$ consists of at most two half-open segments $I^{l}$ and $I^{r}$; one to the left of $a$, and one to the right from $b$. Since there are no points of $B\cap \partial A$ to the left of $a$, the segment between $\min (B \cap \ell)$ and $a$ -- including $I^{l}$ -- lies entirely in $A$, or in $A^{c}$. In the latter case we are done. In the former case, the definition of $I$ implies that $\min I \geq a$, and hence $I^{l} = \emptyset$. A similar reasoning shows that $I^{r} \subset A^{c}$.

Now, we may combine \eqref{form10} and \eqref{form11} to deduce
 \begin{align*} \NC_{B}(A^c,\ell) \leq \int_{B \cap \ell} |\chi_{A^{c}} - \chi_{I}| \, d\calH^{1} \leq \calH^{1}(A \cap J), \end{align*}
which concludes the proof of the lemma.
\end{proof}

\begin{lemma}\label{lemma1} Let $A \subset \Hek$, let $B$ be a ball, and let $\ell \in \calL$ such that $A \cap \ell$ is $\calH^{1}$ measurable. Then,
\begin{displaymath} \NM_{B}(A,\ell) \leq \wid_{B}(\partial A,\ell). \end{displaymath}
\end{lemma}

\begin{proof}
 Assume that $\ell \in \calL(B\cap \partial A)$, otherwise the lemma is trivial. Pick any orientation for $\ell$ and write $J$ for the segment from $a:=\min (B \cap \partial A \cap \ell)$ to $b:=\max (B \cap \partial A \cap \ell)$. Since $\ell\in\calL$ is a horizontal line, we have $\calH^{1}(J) = \mathrm{diam}(J) = \wid_{B}(\partial A,\ell)$. Then apply Lemma \ref{l:convex} to $A$ and $A^c$. Since $\partial A = \partial A^c$, we obtain
 \begin{equation*}\label{eq:NC_bound_J}
  \NC_{B}(A^c,\ell) \leq \calH^1(A \cap J)\quad\text{and}\quad
   \NC_{B}(A,\ell) \leq \calH^1(A^c \cap J),
 \end{equation*}
and hence,
 \begin{displaymath}
  \NM_{B}(A,\ell) =  \NC_{B}(A,\ell) + \NC_{B}(A^c,\ell) \leq \calH^1(J) = \wid_{B}(\partial A,\ell),
 \end{displaymath}
as claimed. \end{proof}

Next, following~\cite{CKN}, we define the \emph{non-monotonicity} of a measurable set $A \subset \Hek$ in a ball $B$ of radius $r > 0$ by
\begin{equation} \label{NMDef}
\NM_{B}(A) := \frac{1}{r^{2k+2}} \int_{\calL} \NM_{B}(A,\ell) \, d\h(\ell)  .
\end{equation}
The quantity $\NM_{B}(A)$ is invariant under left translations and scaling (of both $B$ and $A$ simultaneously). We remark, leaving the details to the reader, that if $A \subset \Hek$ is measurable, then $A \cap \ell$ is $\calH^{1}$ measurable for $\h$ almost every $\ell \in \calL$, so $\NM_{B}(A,\ell)$ is well defined $\h$ almost surely, and the map $\ell\in \calL \mapsto \NM_{B}(A,\ell)$ is $\h$-measurable.

In a similar fashion, for a closed set $E \subset \Hek$, we define the \emph{horizontal width} of $E$ in a ball $B$ of radius $r > 0$ by
\begin{equation}\label{e:def-totalwidth}
\wid_{B}(E) := \frac{1}{r^{2k+2}} \int_{\calL} \wid_{B}(E,\ell) \, d\h(\ell),
\end{equation}
We justify briefly the $\h$-measurability of the map $\ell \in \calL \mapsto \wid_{B}(E,\ell)$. Since we assume that $E$ is closed, the intersection $E \cap B$ is compact. Hence, for any $n \in \N$, we can cover $E \cap B$ by finitely many balls of radius $1/n$. If the union of these balls is denoted by $E_{n}$, then it is clear that $\ell \in \calL \mapsto \wid_{B}(E_{n},\ell)$ is $\h$-measurable, and $\wid_{B}(E_{n},\ell) \searrow \wid_{B}(E,\ell)$ for any fixed line $\ell \in \calL$, by the compactness of $E \cap B$. Hence $\ell \in \calL \mapsto \wid_{B}(E,\ell)$ is $\h$-measurable.

By Lemma \ref{lemma1}, every measurable set $A \subset \Hek$ satisfies $\NM_{B}(A) \leq \wid_{B}(\partial A)$. We apply now this inequality to the connected components of the complement of a given closed set $E$, noting that their boundaries are contained in $E$.

\begin{proposition}\label{form1}
Given a closed set $E\subset \Hek$, one has
\begin{equation*} \NM_{B}(\Omega) \leq \wid_{B}(E) \,\,\text{ for all connected  components } \Omega \text{ of } E^c. \end{equation*}
\end{proposition}

Next, we recall that, given $\ell \in \calL$, a set $I\subset \ell$ is said to be a \emph{monotone} subset of $\ell$ if its characteristic function is a monotone function on $\ell$, up to a  null set. Equivalently, up to a null set, $I$ and $\ell \setminus I$ are intervals. As the terminology suggests, the non-monotonicity of a set $A\subset \Hek$ with respect to a horizontal line $\ell\in\calL$ on a ball $B$ gives a way to measure how $A\cap \ell$ differs from being a monotone subset of $\ell$ inside $B$, see Lemma~\ref{lem:non-monotonicity-vs-monotone-sets}. A set $A\subset \Hek$ is said to be \emph{monotone} if for $\h$ almost all $\ell \in\calL$, the intersection $A\cap \ell$ is a monotone subset of $\ell$. These definitions are due to Cheeger and Kleiner~\cite{CK}. For the sake of completeness, we recall below the classification of monotone sets in $\Hek$ proved in~\cite{CK} for $k=1$ and in~\cite{NY} for $k\geq 2$.

\begin{thm}\cite[Theorem~5.1 ($k=1$)]{CK} \cite[Proposition~65 ($k \geq 2$)]{NY} \label{thm:classification-monotone-sets}
If a measurable set $A \subset \Hek$ is monotone, then, up to an $\calH^{2k+2}$ null set, either $A=\emptyset$, $A=\Hek$ or $A$ is a half-space.
\end{thm}

Here a half-space denotes an open subset of $\Hek$ whose boundary is an affine hyperplane when identifying $\Hek$ with $\Rk\times\R$ as a real vector space.

\begin{remark}  \label{rmk:non-monotonicity}

A slightly different definition of non-monotonicity, denoted by $\widetilde{\text{NM}}_{B}(A,\ell)$ below, is given in~\cite{NY}. Namely, given $\ell \in \calL$, a measurable set $A \subset \Hek$ such that $A \cap \ell$ is $\calH^{1}$ measurable, and a ball $B \subset \Hek$,
\begin{displaymath} \widetilde{\text{NM}}_{B}(A,\ell) := \inf \left\{ \int_{B \cap \ell} |\chi_{A} - \chi_{I}| \, d\calH^{1} : I \text{ is a  monotone subset of } \ell\right\}, \end{displaymath}
and
\begin{equation*}
\widetilde{\text{NM}}_{B}(A) := \frac{1}{r^{2k+2}} \int_{\calL} \widetilde{\text{NM}}_{B}(A,\ell) \, d\h(\ell)
\end{equation*}
where $r>0$ denotes the radius of $B$. As an immediate consequence of the definitions, one has $\NC_B(A,l) \leq \widetilde{\text{NM}}_{B}(A,\ell)$ and $\widetilde{\text{NM}}_{B}(A,\ell) =\widetilde{\text{NM}}_{B}(A^c,\ell)$, hence $\NM_B(A) \leq 2\, \widetilde{\text{NM}}_{B}(A)$. Conversely, Lemma~\ref{lem:non-monotonicity-vs-monotone-sets} below shows that $\widetilde{\text{NM}}_{B}(A) \leq \NM_B(A)$. Hence the two notions of non-monotonicity are comparable.

\end{remark}

\begin{lemma} \label{lem:non-monotonicity-vs-monotone-sets} Let $A\subset \Hek$, let $B$ be a ball, and let $\ell\in\calL$ such that $A \cap \ell$ is $\calH^{1}$ measurable. For every $\epsilon >0$, there is a monotone subset $I$ of $\ell$ such that
\begin{equation*} \label{e:non-monotonicity-vs-monotone-sets}
\int_{B \cap \ell} |\chi_{A} - \chi_{I}| \, d\calH^{1} \leq \NM_B(A,\ell) + \epsilon.
\end{equation*}
\end{lemma}

\begin{proof} Assume that $\ell \in \calL(B)$, otherwise $I$ can obviously be taken to be the empty set. Let $I_1, I_2 \subset \ell$ be intervals such that
\begin{equation*}
\int_{B \cap \ell} |\chi_{A} - \chi_{I_1}| \, d\calH^{1} \leq \NC_B(A,\ell) + \epsilon / 2,
\end{equation*}
\begin{equation*}
\int_{B \cap \ell} |\chi_{A^c} - \chi_{I_2}| \, d\calH^{1} \leq \NC_B(A^c,\ell) + \epsilon / 2.
\end{equation*}
Pick any orientation for $\ell$ and let $a:=\min (B \cap \ell)$, $b:=\max (B \cap \ell)$. If $a\in I_1$, then $I:= (-\infty, \max I_1]$ is a monotone subset of $\ell$ with $I\cap B = I_1 \cap B$. Hence $I$ gives the required conclusion by the choice of $I_1$. Similarly, if $b\in I_1$, then $I:=[\min I_1,+\infty)$ gives the required the conclusion. If $a\in I_2$, then $J:=(-\infty, \max I_2]$ is a monotone subset of $\ell$ with $J\cap B = I_2 \cap B$. Hence, setting $I:=J^c$, we get
\begin{equation*}
\int_{B \cap \ell} |\chi_{A} - \chi_{I}| \, d\calH^{1} =  \int_{B \cap \ell} |\chi_{A^c} - \chi_{I_2}| \, d\calH^{1},
\end{equation*}
and $I$ gives the required conclusion by the choice of $I_2$. Similarly, if $b\in I_2$, then $I:=(-\infty, \min I_2]$ gives the required conclusion.
It remains now to consider the case where $I_1$ and $I_2$ are strict non-empty subintervals of $B\cap \ell$. If $\min I_1 \leq \min I_2$, see Figure \ref{fig5}, we set $I:=(-\infty,\max I_1]$. We have $[a,\min I_1) \subset I_2^c \cap I$, hence,
\begin{equation*}
\begin{split}
\int_{B \cap \ell} |\chi_{A} - \chi_{I}| \, d\calH^{1} &\leq \int_{[a,\min I_1]} |\chi_{A} - \chi_{I_2^c}|\, d\calH^{1} + \int_
{[\min I_1,b]} |\chi_{A} - \chi_{I_1}|\, d\calH^{1}\\
& \leq \int_{B \cap \ell} |\chi_{A^c} - \chi_{I_2}| \, d\calH^{1} + \int_
{B \cap \ell} |\chi_{A} - \chi_{I_1}|\, d\calH^{1}\\
& \leq \NM_B(A,\ell) + \epsilon
\end{split}
\end{equation*}
by the choice of $I_1$ and $I_2$.
\begin{figure}[h!]
\begin{center}
\includegraphics[scale = 0.6]{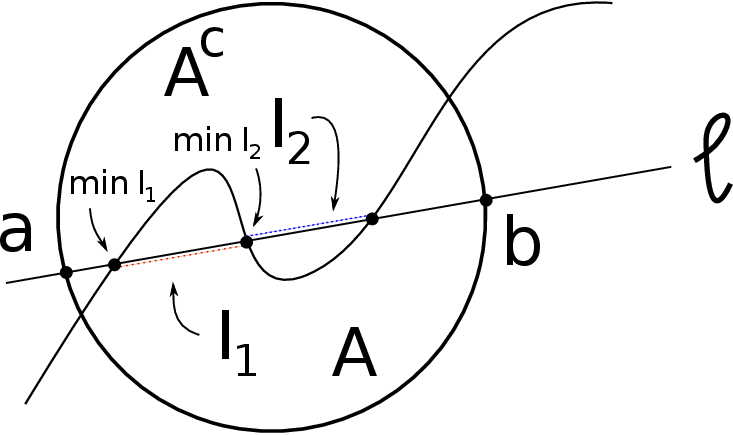}
\caption{The intervals $I_{1}$ and $I_{2}$.}\label{fig5}
\end{center}
\end{figure}
If $\min I_2 < \min I_1$, we set $I:=[\max I_2,+\infty)$. We have $[a,\min I_2) \subset I_1^c \cap I^c$, hence,
\begin{equation*}
\begin{split}
\int_{B \cap \ell} |\chi_{A} - \chi_{I}| \, d\calH^{1} &\leq \int_{[a,\min I_2]} |\chi_{A} - \chi_{I_1}| \, d\calH^{1} + \int_{[\min I_2,b]} |\chi_{A} - \chi_{I_2^c}|\, d\calH^{1} \\
& \leq \int_{B \cap \ell} |\chi_{A} - \chi_{I_1}| \, d\calH^{1} + \int_{B \cap \ell} |\chi_{A^c} - \chi_{I_2}| \, d\calH^{1}\\
& \leq \NM_B(A,\ell) + \epsilon
\end{split}
\end{equation*}
by the choice of $I_1$ and $I_2$.
\end{proof}


\subsection{The approximation} \label{sect:small-width}
The main goal of this section is Corollary~\ref{cor:main}, which roughly states that a Semmes surface with small width in a given ball can be bilaterally approximated by a hyperplane in a slightly smaller ball.

Corollary~\ref{cor:main} is obtained by combining Propositions ~\ref{prop:small-non-mono-imply-approx-half-space} and ~\ref{monotonicityToBetas}. The former states that a set with upper Ahlfors-regular boundary and small non-mo{\-}no{\-}to{\-}nicity in some ball is measure-theoretically close to a half-space in a slightly smaller ball. This result is due to Naor and Young~\cite[Proposition~66]{NY}. It is inspired by a deep result of Cheeger, Kleiner, and Naor,~\cite[Theorem 4.3]{CKN}.
Proposition~\ref{prop:small-non-mono-imply-approx-half-space} is less quantitative than \cite[Theorem 4.3]{CKN}, but sufficient for our purposes.
The statement of Proposition~\ref{prop:small-non-mono-imply-approx-half-space} differs slightly from~\cite[Proposition~66]{NY}, however, taking into account Remark~\ref{rmk:non-monotonicity}, the proof given in~\cite{NY} can be verbatim rephrased to give a proof of Proposition~\ref{prop:small-non-mono-imply-approx-half-space} as stated here. The argument uses indeed only upper Ahlfors-regularity rather than Ahlfors-regularity of the boundary and is based on Theorem~\ref{thm:classification-monotone-sets} together with a compactness argument.

A \emph{hyperplane} in $\He^k$ is  an affine hyperplane when identifying $\Hek$ with $\Rk\times\R$ as a real vector space.
In the rest of the paper, a hyperplane
will be typically denoted by $P$. A hyperplane is either \textit{vertical}, which means that there is $\nu\in\Sk$ such that $P = p\cdot \W_\nu$ for every $p\in P$, or \textit{horizontal}, which means that there is a unique $p\in P$ such that $P=p\cdot H$ where $H:=\Rk\times\{0\}$. Given a hyperplane $P$, the two half-spaces with boundary $P$ will be denoted by $P^{-}$ and $P^{+}$.

In the next statement, we say that a set is \emph{$C$-upper Ahlfors-regular} if the set is upper Ahlfors-regular (with dimension $2k+1$) and $C$ is a positive constant for which~\eqref{e:upper-regular} holds with $s = 2k + 1$.

\begin{proposition} \cite[Proposition~66]{NY} \label{prop:small-non-mono-imply-approx-half-space}
For every $C>0$ and $\delta>0$, there exists $0< \gamma < 1$ such that the following holds. If $F \subset \Hek $ is measurable with $C$-upper Ahlfors-regular boundary, $p\in \Hek$, $r>0$, and $\NM_{B(p,r)}(F) \leq  \gamma^{2k+3}$, then there is a half-space $P^-\subset \Hek$ such that
\begin{equation}\label{form37} \frac{\calH^{2k+2}([F \bigtriangleup P^{-}] \cap B(p,\gamma r))}{\calH^{2k+2}(B(p,\gamma r))} \leq \delta. \end{equation}
\end{proposition}

Thanks to Proposition~\ref{form1}, we infer that if a closed upper Ahlfors-regular set has small horizontal width in some ball, then one can apply Proposition~\ref{prop:small-non-mono-imply-approx-half-space} to  every connected component of its complement and one gets the following corollary.

\begin{cor} \label{cor:small-width-imply-cc-close-half-spaces} Let $S\subset \Hek$ be a closed upper Ahlfors-regular set. For every $\delta>0$, there is $0< \gamma < 1$, depending only on $\delta,k$ and the upper Ahlfors-regularity constant for $S$, such that the following holds. If $p\in \Hek$, $r>0$, and $\wid_{B(p,r)}(S) \leq  \gamma^{2k+3}$, then, for every connected component $\Omega$ of $S^c$, there is a half-space $P_\Omega^-\subset \Hek$ such that
\begin{equation} \frac{\calH^{2k+2}([\Omega \bigtriangleup P_\Omega^{-}] \cap B(p,\gamma r))}{\calH^{2k+2}(B(p,\gamma r))} \leq \delta. \end{equation}
\end{cor}

Condition~\textbf{B} first appears in Proposition~\ref{monotonicityToBetas} below. It states that a closed set satisfying Condition~\textbf{B} whose complementary components are measure-theoretically close to half-spaces in some ball is itself close to a hyperplane in a slightly smaller ball.

\begin{proposition}\label{monotonicityToBetas} There is a dimensional constant $\overline \epsilon >0$ such that the following holds for every $0 <\epsilon < \overline \epsilon$. Assume that $S \subset \Hek$ is a closed set satisfying Condition~\textup{\textbf{B}}. There exists $\delta > 0$, depending only on $\epsilon,k$ and the Condition~\textup{\textbf{B}} constant for $S$, such that if $p \in S$ and $0<r <\diam S$ are such that for every connected component $\Omega$ of $S^c$, there exists a half-space $P_{\Omega}^{-} \subset \Hek$ with
\begin{equation} \label{form3}
\frac{\calH^{2k+2}([\Omega \bigtriangleup P_{\Omega}^{-}] \cap B(p,r))}{\calH^{2k+2}(B(p,r))} \leq \delta,
\end{equation}
 then, there exists a hyperplane $P \subset \Hek$ such that $\dist(q,P) \leq \epsilon r$ for all $q \in S \cap B(p,r/80)$ and $\dist(q,S) \leq \epsilon r$ for all $q\in P\cap B(p,r/80)$.
\end{proposition}

\begin{proof} Fix $\epsilon > 0$, which we may assume to be sufficiently small, depending only on $k$, for the following arguments to work. Left translations and dilations send connected components to connected components, hyperplanes to hyperplanes, and half-spaces to half-spaces, and also preserve the relative size of balls from Condition~\textbf{B}, so we may assume that $p= 0 \in S$ and $r = 1$. We will prove that there is a universal constant $c>0$ such that, for every $\epsilon>0$ fixed small enough, depending only on $k$, there exists a hyperplane $P \,(=P_1) \subset \Hek$ such that $\dist(q,P) \leq c \epsilon$ for all $q \in S \cap B(0,1/40)$ and $\dist(q,S) \leq 2c \epsilon$ for all $q\in P \cap B(0,1/80)$, provided $\delta$ in~\eqref{form3} is chosen small enough, see~\eqref{e:S-closed-to-P1} and~\eqref{e:P-closed-to-S}.

We start with the following observation. For every connected component $\Omega$ of $S^c$, points in $S \cap P_{\Omega}^{-} \cap B(0,1/2)$ must lie close to the boundary $P_{\Omega}$ of $P_{\Omega}^-$. Indeed, pick $q \in S \cap P_{\Omega}^{-} \cap B(0,1/2)$ and assume that $\dist(q,P_{\Omega}) > \epsilon$. Since $q \in S$, we may apply Condition~\textbf{B} to the ball $B(q,\epsilon) \subset P_{\Omega}^{-} \cap B(0,1)$ to find inside $B(q,\epsilon)$ two balls in different connected components of $S^c$ with radii $\sim_{reg} \epsilon$. In particular, one of the balls, say $B$, lies in $[P_{\Omega}^{-} \setminus \Omega] \cap B(0,1)$ which gives the lower bound
\begin{displaymath}
\calH^{2k+2}([\Omega \bigtriangleup P_{\Omega}^{-}] \cap B(0,1)) \geq \calH^{2k+2}([P_{\Omega}^{-} \setminus \Omega] \cap B(0,1)) \geq \calH^{2k+2}(B) \gtrsim_{reg} \epsilon^{2k+2}.
\end{displaymath}
So, taking $\delta$ much smaller than $\epsilon^{2k+2}$ in~\eqref{form3}, this gives a contradiction. Hence, for every connected component $\Omega$ of $S^c$,  we have
\begin{equation} \label{e:P^-}
\dist(q,P_{\Omega}) \leq \epsilon \quad \text{for all } q\in S \cap P_{\Omega}^- \cap B(0,1/2).
\end{equation}

Next, we would like to find one connected component $\Omega$ of $\Hek \setminus S$ for which we can reach a similar conclusion for points $q \in S \cap P_{\Omega}^{+} \cap B(0,1/40)$. To this end, we start by singling out two connected components $\Omega_{1}$, $\Omega_{2}$ of $S^c$ in the following way. Let $0< \rho <1$ be a suitable constant, to be chosen small later depending only on $\epsilon$. Then, apply Condition~\textbf{B} to the ball $B(0,\rho)$ to find two distinct connected components $\Omega_{1},\Omega_{2}$ of $S^c$, and two balls
\begin{displaymath}
B_{1} \subset \Omega_{1} \cap B(0,\rho ) \quad \text{and} \quad B_{2} \subset \Omega_{2} \cap B(0,\rho)
\end{displaymath}
with radii $\sim_{reg} \rho$. Then, let $P_{1}^-,P_{2}^-$ be the half-spaces, with boundary $P_1$, $P_2$ respectively, associated to $\Omega_{1}$ and $\Omega_{2}$ as in~\eqref{form3}, namely,
\begin{equation} \label{form3-bis}
\calH^{2k+2}([\Omega_{1} \bigtriangleup P_{1}^{-}] \cap B(0,1)) \lesssim_k \delta \quad \text{and} \quad \calH^{2k+2}([\Omega_{2} \bigtriangleup P_{2}^{-}] \cap B(0,1)) \lesssim_k \delta.
\end{equation}
If $\delta$ is small enough, we claim that both hyperplanes $P_{1}$ and $P_{2}$ intersect $B(0,\rho)$. To see this, assume for instance that this fails for $P_{1}$. Then both balls $B_{1},B_{2}$ lie either in $P_{1}^{+}$ or $P_{1}^{-}$. Both cases lead to a contradiction. First, if $B_{1}$ lies in $P_{1}^{+}$, then $B_{1} \subset \Omega_{1} \setminus P_{1}^{-} \subset \Omega_{1} \bigtriangleup P_{1}^{-}$, which violates the first part of~\eqref{form3-bis} for $\delta$ much smaller than $\rho^{2k+2}$, since $\calH^{2k+2}(B_{1}) \sim_{reg} \rho^{2k+2}$. Similarly, if $B_{2}$ lies in $P_{1}^{-}$, then $B_{2} \subset P_{1}^{-} \setminus \Omega_{1} \subset \Omega_{1} \bigtriangleup P_{1}^{-}$, which once again violates the first part of~\eqref{form3-bis}.

Now let $B_\text{Euc}$ denote the Euclidean ball centred at the origin with radius $1/20$ so that
\begin{displaymath}
B(0,1/40) \subset B_\text{Euc} \subset B(0,1/2).
\end{displaymath}
We claim that the hyperplanes $P_{1}$ and $P_{2}$ lie very close to each other inside $B_\text{Euc}$, provided $\rho$ was chosen small enough compared to $\epsilon$ . We quantify this by claiming that $P_{2} \cap B_\text{Euc}$ lies in the closed Euclidean $\epsilon^{2}$-neighbourhood $P_{1,\text{Euc}}(\epsilon^{2})$ of $P_{1}$, that is,
\begin{equation}\label{form7}
P_{2} \cap B_\text{Euc} \subset P_{1,\text{Euc}}(\epsilon^{2}).
\end{equation}
The argument to prove~\eqref{form7} is completely Euclidean. To see what is going on, it is helpful to first visualise what happens if the hyperplanes $P_{1},P_{2}$ both contain $0$. In this case, if $P_{2} \cap B_\text{Euc} \not\subset P_{1,\text{Euc}}(\epsilon^{2})$, then the intersection $P_{1}^{-} \cap P_{2}^{-} \cap B_\text{Euc}$ contains a Euclidean ball $B'_\text{Euc}$ with radius $\sim_k \epsilon^{2}$, see Figure~\ref{fig4}.
\begin{figure}[h!]
\begin{center}
\includegraphics[scale = 0.5]{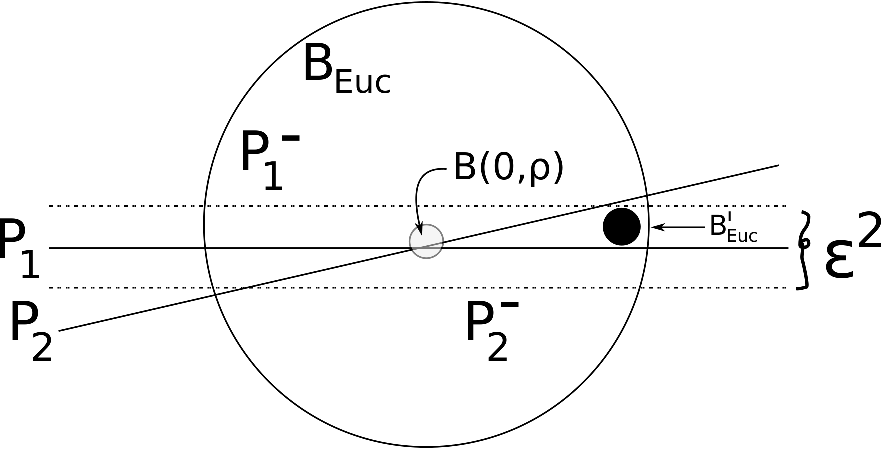}
\caption{The planes $P_{1},P_{2}$ and the $\epsilon^{2}$-Euclidean ball $B'_\text{Euc} \subset P_{1}^{-} \cap P_{2}^{-} \cap B_\text{Euc}$ drawn in black.}\label{fig4}
\end{center}
\end{figure}
Then, we note that the same remains true if $P_{1},P_{2}$ intersect the small ball $B(0,\rho)$, provided $\rho$ was chosen sufficiently small compared to $\epsilon$. We indeed start by translating both planes in the Euclidean sense by $\leq c\rho$ for some universal constant $c>0$ so that they contain $0$. This is possible since $d_\text{Euc}(p',q') \leq c  d(p',q')$ for $p', q' \in B(0,1/2)$ and for some universal constant $c>0$. Then we find a Euclidean ball with radius $\sim_k \epsilon^2$ as above, and finally shift back, making the ball a bit smaller if necessary. Since $\calH^{2k+2}$ coincides, up to a multiplicative constant, with the $(2k+1)$-dimensional Lebesgue measure on $\Rk\times\R$, we have $\calH^{2k+2} (B'_\text{Euc}) \sim_k \epsilon^{4k+2}$ and we get
\begin{equation*}
\calH^{2k+2}(P_{1}^{-} \cap P_{2}^{-} \cap B(0,1)) \gtrsim_k \epsilon^{4k+2}.
\end{equation*}
But if $\delta$ is chosen much smaller than $\epsilon^{4k+2}$, this contradicts~\eqref{form3-bis}. Indeed, writing $A := P_{1}^{-} \cap P_{2}^{-} \cap B(0,1)$, we have
\begin{equation} \label{form38}
\begin{split}
\calH^{2k+2}(A) & = \calH^{2k+2}(A \cap \Omega_{1}) + \calH^{2k+2}(A \setminus \Omega_{1})\\
& \leq \calH^{2k+2}([P_{2}^{-} \setminus \Omega_{2}] \cap B(0,1)) + \calH^{2k+2}([P_{1}^{-} \setminus \Omega_{1}] \cap B(0,1))\\
& \leq \calH^{2k+2}([\Omega_{2} \bigtriangleup P_{2}^{-}] \cap B(0,1)) + \calH^{2k+2}([\Omega_{1} \bigtriangleup P_{1}^{-}] \cap B(0,1)) \lesssim_k \delta.
\end{split}
\end{equation}
Hence~\eqref{form7} holds. As a corollary, we now show that
\begin{equation}\label{form8}
P_{1}^{+} \cap P_{2}^{+} \cap B_\text{Euc} \subset P_{1,\text{Euc}}(\epsilon^{2}).
\end{equation}
To see this, let $U_{1}$ and $U_{2}$ be the two connected components of $B_\text{Euc} \setminus P_{1,\text{Euc}}(\epsilon^{2})$, labelled so that $U_1 \subset P_1^-$ and $U_2 \subset P_1^+$, see Figure \ref{fig3}. By~\eqref{form7}, the hyperplane $P_2$ also separates $U_1$ and $U_2$, hence, either $U_1 \subset P_2^-$ and $U_2 \subset P_2^+$, or $U_1 \subset P_2^+$ and $U_2 \subset P_2^-$.
\begin{figure}[h!]
\begin{center}
\includegraphics[scale = 0.5]{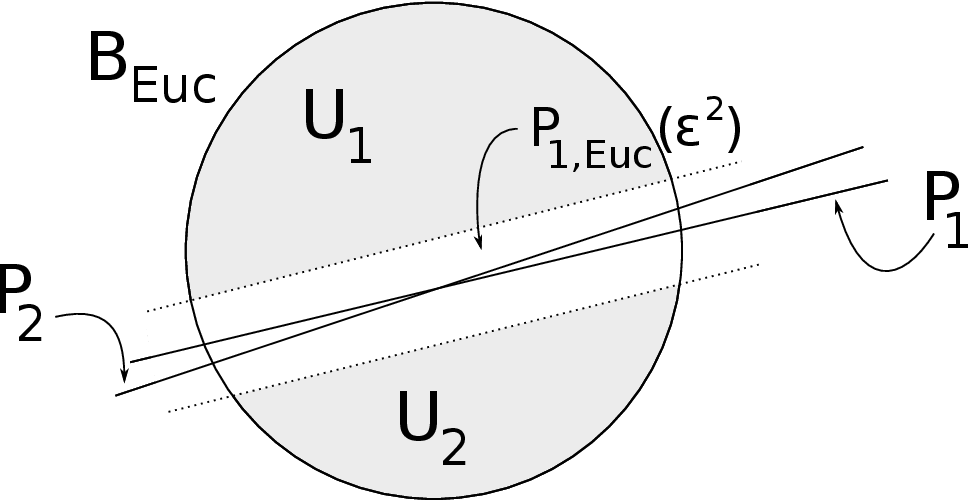}
\caption{The planes $P_{1},P_{2}$ and the components $U_{1},U_{2}$.}\label{fig3}
\end{center}
\end{figure}
Arguing by contradiction, if $P_{1}^{+} \cap P_{2}^{+} \cap B_\text{Euc} \not\subset P_{1,\text{Euc}}(\epsilon^{2})$, then $P_{1}^{+} \cap P_{2}^{+} \cap B_\text{Euc}$ meets $U_{2}$. In particular, $U_2 \cap P_2^+ \not= \emptyset$ and it follows that $U_2 \subset P_2^+$ and $U_1 \subset P_2^-$. Hence $U_{1} \subset P_{1}^{-} \cap P_{2}^{-} \cap B_\text{Euc}$. On the other hand, assuming that $\epsilon$ is sufficiently small, then $\calH^{2k+2}(U_{1}) \sim_k 1$ and we get that $\calH^{2k+2}(P_{1}^{-} \cap P_{2}^{-} \cap B(0,1)) \gtrsim_k 1$. Choosing $\delta$ small enough, this contradicts~\eqref{form38} and hence~\eqref{form8} holds.

Now we can conclude the proof of the proposition. Recall that we already know by~\eqref{e:P^-} applied to $\Omega_1$ that $\dist(q,P_{1}) \leq \epsilon$ for all $q \in S \cap P_{1}^{-} \cap B(0,1/2)$. Next, pick $q \in S \cap P_{1}^{+} \cap B(0,1/40)$. First, if $q \in P_{2}^{+}$ then $q \in P_{1}^{+} \cap P_{2}^{+} \cap B_\text{Euc} \subset P_{1,\text{Euc}}(\epsilon^{2})$ by~\eqref{form8}, hence $\dist(q,P_{1}) \leq c' \epsilon$ for some universal constant $c'>0$. Indeed recall that $P_{1,\text{Euc}}(\epsilon^{2})$ refers to the Euclidean $\epsilon^{2}$-neighbourhood of $P_1$ and that $d(p',q') \leq c' d_{\text{Euc}}(p',q')^{1/2}$ for $p',q' \in B(0,1/2)$ and for some universal constant $c'>0$. The same obviously still works if $q \in P_{2}$ by~\eqref{form7}. Finally, if $q \in S \cap P_{2}^{-} \cap B(0,1/40)$, then $\dist(q,P_{2}) \leq \epsilon$ by~\eqref{e:P^-} applied to $\Omega_2$. But $P_{2} \cap B_\text{Euc} \subset P_{1,\text{Euc}}(\epsilon^{2})$, hence $\dist(q,P_{1}) \leq (1+c')\epsilon$. So we have now proven that
\begin{equation} \label{e:S-closed-to-P1}
S \cap B(0,1/40) \subset P_1(c\epsilon)
\end{equation}
for some universal constant $c>0$ and where $P_1(c\epsilon)$ denotes the closed $c\epsilon$-neighbourhood of $P_1$ in the Kor\'anyi metric. In other words, $S$ is close to $P_1$ inside $B(0,1/40)$.

We show now that also $P_1$ is close to $S$ inside a slightly smaller ball, namely,
\begin{equation} \label{e:P-closed-to-S}
P_1 \cap B(0,1/80) \subset S(2c\epsilon)
\end{equation}
where $S(2c\epsilon)$ denotes the closed $2c\epsilon$-neighbourhood of $S$ in the Kor\'anyi metric. This will complete the proof of the proposition choosing $P:=P_1$.

We denote by $U_1'$ and $U_2'$ the two connected components of $B(0,1/40) \setminus P_1(c\epsilon)$, labelled so that $U_1' \subset P_1^-$ and $U_2'\subset P_1^+$. Since $U_1'$ is connected and does not meet $S$ by~\eqref{e:S-closed-to-P1}, it is contained in some connected component of $S^c$, that must be $\Omega_1$ by the first part of~\eqref{form3-bis}. Indeed, otherwise $U_1' \subset [P_1^- \setminus \Omega_1] \cap B(0,1)$ and since $\calH^{2k+2} (U_1') \sim_k 1$, at least if $\epsilon$ is sufficiently small, depending only on $k$, this contradicts the first part of~\eqref{form3-bis} choosing $\delta$ small enough. Hence $U_1' \subset \Omega_1$. By~\eqref{form7} and the choice of the universal constant $c$, we have $P_2 \cap B(0,1/40) \subset P_{1,\text{Euc}}(\epsilon^{2}) \subset P_1(c\epsilon)$, hence either $U_1' \subset P_2^+$ or $U_1' \subset P_2^-$. We have that $U_1' \subset P_2^+$, because otherwise $U_1' \subset [P_2^- \setminus \Omega_2] \cap B(0,1)$ and this contradicts now the second part of~\eqref{form3-bis}. Since $P_2$ separates $U_1'$ from $U_2'$, it follows that $U_2'\subset P_2^-$. Then, arguing in a similar way as we did for $U_1'$ and using the second part of~\eqref{form3-bis}, we get that $U_2' \subset \Omega_2$. In particular $U'_1$ and $U'_2$ are contained in different connected components of $S^c$. Going back to the proof of~\eqref{e:P-closed-to-S}, we argue by contradiction and assume that there is $q\in P_1 \cap B(0,1/80)$ with $\dist(q,S) > 2c\epsilon$. Then, choosing $\epsilon$ small enough, we get that $B(q,2c\epsilon) \subset B(0,1/40) \setminus S$. On the other hand, $B(q,2c\epsilon)$ meets both $U_1'$ and $U_2'$, and since $U_1'$ and $U_2'$ are contained in different connected components of $S^c$, the ball $B(q,2c\epsilon)$ should meet $S$, a contradiction.
\end{proof}

Combining Corollary~\ref{cor:small-width-imply-cc-close-half-spaces} and Proposition~\ref{monotonicityToBetas} shows that small horizontal width for a Semmes surface implies flatness.

\begin{cor} \label{cor:main}
There is a dimensional constant $\overline \epsilon >0$ such that the following holds for every $0 <\epsilon < \overline \epsilon$. Let $S\subset \Hek$ be a Semmes surface. There is $0<\gamma<1$, depending only on $\epsilon,k$ and on the upper Ahlfors-regularity and Condition~\textbf{B} constants for $S$, such that the following holds. If $p\in S$, $0<r<\diam S$, and $\wid_{B(p,r)}(S) \leq (80\gamma)^{2k+3}$, then there is a hyperplane $P\subset \Hek$ such that
\begin{equation*}
\sup_{q\in S \cap B(p, \gamma r)} \dist(q,P) + \sup_{q\in P \cap B(p, \gamma r)} \dist(q,S) \leq \epsilon \gamma r.
\end{equation*}
\end{cor}


\section{Lower Ahlfors-regularity and big vertical projections} \label{BVPSection}
We first prove in this section that closed subsets of $\Hek$ satisfying condition~\textbf{B}, and therefore Semmes surfaces, are lower Ahlfors-regular.

\begin{proposition}[Lower Ahlfors-regularity for sets satisfying condition~B] \label{prop:lowerAR} Let $S \subset \Hek$ be a  closed set satisfying condition~\textbf{B}. Then $S$ is lower Ahlfors-regular with a lower Ahlfors-regularity constant depending only on $k$ and the Condition~\textbf{B} constant of $S$.
\end{proposition}

\begin{proof}
The proof follows from the relative isoperimetric inequality in $\Hek$. Recall indeed that there is $\lambda\geq 1$ such that, given $E \subset \Hek$ measurable, $p\in \Hek$, and $r>0$, we have
\begin{equation} \label{e:relative-isop}
\min \{\calH^{2k+2} (B(p,\lambda^{-1} r)\cap E), \calH^{2k+2} (B(p,\lambda^{-1} r)\setminus E)\}^{\frac{2k+1}{2k+2}} \lesssim_{k} \mathcal{H}^{2k+1}(U(p,r) \cap \partial E)~,
\end{equation}
where $U(p,r)$ denotes the open ball with center $p$ and radius $r$. This follows from~\cite[Theorem~1.18]{MR1404326}, recalling that the Kor\'{a}nyi distance is biLipschitz equivalent to the sub-Riemannian distance together with the following fact. If $\mathcal{H}^{2k+1}(U(p,r) \cap \partial E) < +\infty$, then $\perh(E,U(p,r))\lesssim_{k} \mathcal{H}^{2k+1}(U(p,r) \cap \partial E)$ where $\perh$ denotes the  $\He$-perimeter. Now let $S \subset \Hek$ be a  closed set satisfying condition~\textbf{B}, $p\in S$ and $0< r < \diam S$. It follows from condition~\textbf{B} that there are two distinct connected components $\Omega_1$, $\Omega_2$ of $S^c$ such that
\begin{equation*}
\min\{\calH^{2k+2} (B(p, \lambda^{-1} r)\cap \Omega_1), \calH^{2k+2} (B(p,\lambda^{-1} r)\cap \Omega_2)\} \gtrsim_{k} r^{2k+2}.
\end{equation*}
Together with~\eqref{e:relative-isop} we get
\begin{equation*}
\begin{split}
\mathcal{H}^{2k+1}(B(p,r) \cap S) &\geq \mathcal{H}^{2k+1}(U(p,r) \cap \partial \Omega_1) \\
& \geq \min\{\calH^{2k+2} (B(p,\lambda^{-1} r)\cap \Omega_1), \calH^{2k+2} (B(p,\lambda^{-1} r)\setminus \Omega_1)\}^{\frac{2k+1}{2k+2}} \\
& \geq \min\{\calH^{2k+2} (B(p,\lambda^{-1} r)\cap \Omega_1), \calH^{2k+2} (B(p,\lambda^{-1} r)\cap \Omega_2)\}^{\frac{2k+1}{2k+2}} \\
& \gtrsim_{k} r^{2k+1}
\end{split}
\end{equation*}
which concludes the proof of the proposition.
\end{proof}

\begin{remark}  \label{r:lower_ADR} The proof of Proposition~\ref{prop:lowerAR} can be extended, with minor modifications, to get codimension one lower Ahlfors-regularity for closed sets satisfying condition~\textbf{B} in complete doubling metric measure spaces supporting a weak $(1,1)$-Poincar\'e inequality. Variants of the relative isoperimetric inequality hold indeed true in such spaces, see for instance~\cite{MR3165282}, which allow to mimic the proof of Proposition~\ref{prop:lowerAR} in this more general setting.
\end{remark}

We next prove that Semmes surfaces in $\Hek$ have big vertical projections. We first recall the definition of the big vertical projections property.

\begin{definition} [BVP] \label{def:BVP}
We say that a set $E\subset \Hek$ has \textit{big vertical projections} (or BVP in short) if there is  $c > 0$ such that, for all $p \in E$ and $0 < r < \diam E$, there is $\nu\in\Sk$ such that
\begin{equation} \label{e:def-bvp}
\mathcal{H}^{2k+1}(\pi_{\W_\nu}(B(p,r) \cap E)) \geq c\, r^{2k+1}.
\end{equation}
\end{definition}

This rest of this section is devoted to the proof of the following proposition.

\begin{proposition}[BVP for Semmes surfaces] \label{BVPProp}
Let $S \subset \Hek$ be a Semmes surface. Then $S$ has BVP. Moreover, the constant $c$ on the right-hand side of~\eqref{e:def-bvp} can be chosen depending only on $k$ and the upper Ahlfors-regularity and Condition~\textbf{B} constants of $S$.
\end{proposition}

\begin{remark} Note that lower Ahlfors-regularity for a Semmes surface $S \subset \Hek$ can also be recovered as a consequence of Proposition~\ref{BVPProp}. Although projections $\pi_{\W_\nu}$ are not Lipschitz with respect to the Kor\'{a}nyi distance, it however follows from~\cite[Lemma~2.20]{MR3511465}, see also~\cite[Lemma~3.14]{FSS} and~\cite[Lemma~3.6]{CFO}, that, for every $\nu\in\Sk$ and $A\subseteq \Hek$, one has $\calHk (\pi_{\W_\nu}(A)) \lesssim_k \calHk(A)$. Then lower Ahlfors-regularity follows applying this for $A := B(p,r) \cap S$ together with~\eqref{e:def-bvp}.
\end{remark}

We begin the proof of Proposition~\ref{BVPProp}. We will actually prove the following slightly stronger result: we will find a constant $c > 0$, depending only on $k$ and on the upper Ahlfors-regularity and Condition \textbf{B} constants for the Semmes surface $S$, such that, for all $p \in S$ and $0 < r < \diam S$,
\begin{equation} \label{e:bvp-in-average}
\int_{\Sk} \calHk (\pi_{\W_\nu} (B(p,r) \cap S )) \, d\nu \geq c \, r^{2k+1}.
\end{equation}
The expression on the left-hand side of \eqref{e:bvp-in-average} can be thought of as a Heisenberg version of the \emph{Favard length} of $B(p,r) \cap S$. The proof of~\eqref{e:bvp-in-average} for an arbitrary Semmes surface will follow from its validity for hyperplanes, Lemma~\ref{lem:bvp-average-hyperplanes}, together with a compactness argument, Lemma~\ref{lem:BVP-intermediate-step}, and Corollary~\ref{cor:main}.

\begin{lemma} \label{lem:bvp-average-hyperplanes}
There is a dimensional constant $\overline c>0$ such that if $P\subset \Hek$ is a hyperplane, $p\in P$, and $r>0$, then
\begin{equation*}
\int_{\Sk} \calHk (\pi_{\W_\nu} (B(p,r) \cap P )) \, d\nu \geq \overline c \, r^{2k+1}.
\end{equation*}
\end{lemma}

\begin{proof}

It follows from~\eqref{e:h-l(A)} that
$$\int_{\Sk} \calHk (\pi_{\W_\nu} (B(p,r) \cap P )) \, d\nu = \h(\calL(B(p,r) \cap P)).$$
Using  dilations, left translations, and  Lemma \ref{leftInvariance-H2k+1}, we may assume with no loss of generality that $r=1$ and $p=0\in P$.

Next, we use  the notion of $\He$-\emph{perimeter measure} $\perh (E,\cdot)$ of a measurable set $E\subseteq \He^{2k+1}$ to evaluate this expression further. See for instance~\cite{FSSC} for the definition and properties of $\He$-perimeter. We take $E:=P^{-}$ to be one of the two half-spaces bounded by $P$. Since $P$ is of class $C^1$ in the Euclidean sense, $P^{-}$
is of locally finite $\He$-perimeter.
We apply to $P^{-}$ and the open ball $U(0,1)$ the kinematic formula originally due to Montefalcone~\cite{MR2165404}, see also~\cite[(6.1)]{CKN} and~\cite[(112)]{NY}. Specialized to our setting, this ensures the existence of a dimensional constant $0<c_k<+\infty$ such that
\begin{equation*}
\int_{\calL} \per(P^{-}\cap \ell, U(0,1)\cap \ell)  \, d\h(\ell)= c_k \, \perh (P^{-},U(0,1)).
\end{equation*}

On the one hand, we note that
\begin{equation}\label{eq:Perim_h}
\int_{\calL} \per(P^{-}\cap \ell, U(0,1)\cap \ell)\, d\h(\ell) = \h(\calL(U(0,1) \cap P)).
\end{equation}
Indeed, since $P^{-}$ is a half-space bounded by $P$, then $ \per(P^{-}\cap \ell, U(0,1)\cap \ell)= 1$ for horizontal lines $\ell \in \calL$ meeting $P$ transversally inside $U(0,1)$, and $\per(P^{-}\cap \ell, U(0,1)\cap \ell) = 0$ otherwise, that is, if either $\ell \subset P$ or $\ell \not\in \calL(U(0,1) \cap P)$. Since $\h(\{\ell \in \calL: \ell \subset P\}) = 0$, \eqref{eq:Perim_h} follows, and thus
\begin{displaymath}
\h(\calL(U(0,1) \cap P))=c_k \, \perh (P^{-},U(0,1)).
\end{displaymath}

On the other hand, we show that
$\perh(P^{-},U(0,1))$ is uniformly bounded away from zero for all hyperplanes $P$ containing the origin.  Indeed, there exists $\rho >0$ such that $U(0,1)$ contains a Euclidean ball with radius $\rho$, hence $P\cap U(0,1)$ contains a Euclidean $2k$-dimensional ball $B_{\mathrm{Euc}}^{2k}(0,\rho)$ and it follows from~\cite[Corollary~7.7(i)]{FSSC} that
\begin{equation*}
\perh(P^-,U(0,1)) \gtrsim_k  \int_{B_{\mathrm{Euc}}^{2k}(0,\rho)} |C(p) n| \, d\mathcal{H}^{2k}_\mathrm{Euc}(p)
\end{equation*}
where $\mathcal{H}^{2k}_\mathrm{Euc}$ denotes the $2k$-dimensional Hausdorff measure with respect to the Euclidean distance, $n=(n_1,\ldots,n_{2k+1})$ is the Euclidean outward unit normal vector to $P^-$, and
\begin{equation} \label{e:C(p)n}
|C(p)n|=\sqrt{\sum_{i=1}^k \left(n_i - \frac{1}{2} v_{i+k} n_{2k+1}\right)^2 +  \left(n_{k+i}+\frac{1}{2} v_i n_{2k+1}\right)^2}
\end{equation}
for $p=(v,t) \in \Hek$ with $v=(v_1, \dots, v_{2k})\in \Rk$. Denoting by $\mathbb{S}^{2k}$ the Euclidean unit sphere in $\R^{2k+1}$, the map
\begin{displaymath}
n \in \mathbb{S}^{2k} \mapsto \int_{B_{\mathrm{Euc}}^{2k}(0,\rho)} |C(p) n| \, d\mathcal{H}^{2k}_\mathrm{Euc}(p)
\end{displaymath}
is continuous and non-vanishing. Hence, by compactness of $\mathbb{S}^{2k}$, we have
$$\inf_{n\in \mathbb{S}^{2k}} \int_{B_{\mathrm{Euc}}^{2k}(0,\rho)} |C(p) n| \, d\mathcal{H}^{2k}_\mathrm{Euc}(p) >0,$$
and it follows that
\begin{equation} \label{e:lower-ar-hyperplanes}
\perh(P^-,U(0,1)) \gtrsim_k 1.
\end{equation}
This yields
\begin{equation*}
\begin{split}
\int_{\Sk} \calHk (\pi_{\W_\nu} (B(0,1) \cap P )) \, d\nu &\geq \h(\calL(U(0,1) \cap P))\\
&\gtrsim_k \perh(P^-,U(0,1)) \gtrsim_k 1
\end{split}
\end{equation*}
and concludes the proof.
\end{proof}

Estimates similar to \eqref{e:lower-ar-hyperplanes} are also valid in more general Carnot groups, see \cite[Theorem 1.2]{MR2250055}. The next lemma comes as a consequence of Lemma~\ref{lem:bvp-average-hyperplanes} via a compactness argument.

\begin{lemma} \label{lem:BVP-intermediate-step} There is a dimensional constant $\overline\epsilon > 0$ such that the following holds for every $0<\epsilon<\overline \epsilon$. Let $P \subset \Hek$ be a hyperplane, $p\in P$ and $r>0$. Let $U_{1},U_{2}$ denote the two connected components of $B(p,r) \setminus P(\epsilon r)$ where $P(\epsilon r)$ denotes the closed $\epsilon r$-neighbourhood of $P$. Then,
\begin{displaymath}
\h(\calL(U_{1} )\cap \calL( U_{2})) \geq \overline\epsilon \, r^{2k+1}. \end{displaymath}
\end{lemma}

\begin{proof} Using left translations and dilations, we may assume with no loss of generality that $p=0\in P$ and $r=1$. First, for a fixed hyperplane $P\ni 0$, Lemma~\ref{lem:bvp-average-hyperplanes} together with~\eqref{e:h-l(A)} implies the existence of some $\epsilon_{P} > 0$ such that
\begin{equation}\label{form34}
\h(\calL(U_{1}(\epsilon_{P}) )\cap \calL( U_{2}(\epsilon_{P}))) \gtrsim_k 1
\end{equation}
where $U_{1}(\epsilon_{P})$ and $U_{2}(\epsilon_{P})$ are the two connected components of $B(0,1) \setminus P(\epsilon_{P})$. Indeed, each line $\ell\in\calL$ hitting $P$ transversely inside $U(0,1)$, where $U(0,1)$ is the unit open ball centred at the origin, must pass through both $U_{1}(\epsilon_{\ell,P})$ and  $U_{2}(\epsilon_{\ell,P})$ for some $\epsilon_{\ell,P} > 0$ small enough, so \eqref{form34} follows from Lemma~\ref{lem:bvp-average-hyperplanes} by measure-theoretic considerations.

Next, if~\eqref{form34} holds for some hyperplane $P\ni 0$ and some $\epsilon_{P} > 0$, it also holds for all hyperplanes $P'\ni 0$ in an $(\epsilon_{P}/2)$-neighbourhood of $P$, where the latter means that $P'\cap B(0,2) \subset P(\epsilon_{P}/2)$, and for the connected components $U_{1}'(\epsilon_{P}/2)$ and $U_{2}'(\epsilon_{P}/2)$ of $B(0,1) \setminus P'(\epsilon_{P}/2)$, simply using $P'(\epsilon_{P}/2)\cap B(0,1) \subset P(\epsilon_{P}) \cap B(0,1)$ for such $P'$. Finally, the $(\epsilon_{P}/2)$-neighbourhoods form an open cover of the compact set of hyperplanes $\{P \subset \Hek : 0 \in P\}$, so one may pick a finite subcover. Then choosing $\overline\epsilon>0$ less than the minimum of the numbers $\epsilon_{P}/2$ occurring in this finite cover, and less than the dimensional constant implicit in~\eqref{form34} proves the lemma.
\end{proof}

We now turn to the proof of~\eqref{e:bvp-in-average} from which Proposition~\ref{BVPProp} follows.

\begin{proof}[Proof of~\eqref{e:bvp-in-average}]
Using left translations and dilations, we may assume with no loss of generality that $p=0\in S$ and $r=1$. Let $\epsilon>0$ be a suitable constant to be chosen small enough later, depending only on $k$ and the Condition~\textbf{B} constant of $S$. Let $\lambda >0$ be another suitable constant to be chosen small enough, depending on $\epsilon$, $k$, and on the upper Ahlfors-regularity and Condition~\textbf{B} constants of $S$. We argue by contradiction and assume that
\begin{equation} \label{e:proof-bvp-average}
\lambda \geq \int_{\Sk} \calHk (\pi_{\W_\nu} (B(0,1) \cap S)) \, d\nu= \h(\calL(B(0,1) \cap S))
\end{equation}
where the last equality follows from~\eqref{e:h-l(A)}. Recalling the definition~\eqref{e:def-totalwidth} of the horizontal width of $S$ in $B(0,1)$, we get
\begin{equation*}
\wid_{B(0,1)}(S) \leq 2\, \h(\calL(B(0,1) \cap S)) \leq 2 \lambda.
\end{equation*}
Assuming that $\epsilon < \overline \epsilon$ where $\overline \epsilon$ is the dimensional constant given by Corollary~\ref{cor:main}, we now let $0<\gamma<1$ be given by Corollary~\ref{cor:main} applied to the parameter $\epsilon$ and we let $\lambda$ be small enough so that $2\lambda \leq (80 \gamma)^{2k+3}$. Then the previous inequality together with Corollary~\ref{cor:main} imply the existence of a hyperplane $P\subset \Hek$ such that
\begin{equation}\label{form30}
B(0,\gamma) \cap S \subset P(\epsilon\gamma)
\end{equation}
where $P(\epsilon\gamma)$ denotes the closed $\epsilon\gamma$-neighbourhood of $P$.
Replacing $\epsilon$ by $2\epsilon$, we may assume with no loss of generality that $0 \in P$.

To reach a contradiction, we show now that if $S \ni 0$ is a closed set with $\diam S >1$ satisfying Condition \textbf{B} and~\eqref{form30} for some hyperplane $P\ni 0$, then
\begin{equation} \label{form29}
\calL(U_{1})\cap \calL( U_{2}) \subset \calL (B(0,\gamma) \cap S)
\end{equation}
where $U_1$ and $U_2$ denote the two connected components of $B(0,\gamma) \setminus P(\epsilon\gamma)$, provided $\epsilon$ is chosen small enough. To see this, we prove that $U_1$ and $U_2$ are contained in different connected components of $S^c$. Indeed, apply Condition~\textbf{B} to the ball $B(0,\gamma)$ to find two distinct connected components $\Omega_{1}$, $\Omega_{2}$ of $S^c$ and two balls $B_{1} \subset \Omega_1 \cap B(0,\gamma)$ and $B_{2} \subset \Omega_2 \cap B(0,\gamma)$ with radii $\sim_{reg} \gamma$. If $\epsilon > 0$ is chosen small enough, depending on the radii of these balls, it follows that both balls $B_{1}$, $B_{2}$ intersect either $U_{1}$ or $U_{2}$. Moreover, if $B_{1}$ intersects $U_{1}$, say, then $B_{2}$ cannot intersect $U_{1}$. Indeed, otherwise $B_{1}$, $B_{2}$ lie in the same connected component of $S^c$ as a consequence of~\eqref{form30}. Hence $B_{2}$ intersects $U_{2}$. Then it follows, once again from~\eqref{form30}, that $U_{1} \subset \Omega_{1}$ and $U_{2} \subset \Omega_{2}$. Now, fix $\ell \in \calL(U_{1})\cap \calL( U_{2})$, and let $u_{j} \in \ell \cap U_{j}$ for $j \in \{1,2\}$. It follows from the previous argument that the segment $[u_{1},u_{2}]$ has its end points in different connected components of $S^c$ and hence must cross $S$. Moreover, by Euclidean convexity of $B(0,\gamma)$, the segment $[u_{1},u_{2}]$ crosses $S$ inside $B(0,\gamma)$ and this concludes the proof of~\eqref{form29}.

If $\epsilon>0$ was chosen small enough, we then get from Lemma~\ref{lem:BVP-intermediate-step}
\begin{equation*}
\h(\calL(B(0,1) \cap S)) \geq \h(\calL (B(0,\gamma) \cap S)) \geq \h(\calL(U_{1})\cap \calL( U_{2})) \geq \overline\epsilon \, \gamma^{2k+1}
\end{equation*}
for some dimensional constant $\overline \epsilon >0$, which contradicts~\eqref{e:proof-bvp-average} provided $\lambda$ was chosen small enough. This concludes the proof of~\eqref{e:bvp-in-average}.
\end{proof}

\section{The bilateral weak geometric lemma}\label{sect:bwgl}

The main result in this section is the validity of the bilateral weak geometric lemma for vertical hyperplanes, Definition~\ref{def:bwgl-vertical-hyperplanes}, for Semmes surfaces in $\Hek$, see Proposition~\ref{prop:bwgl-vertical-hyperplanes}.

\subsection{Width is integrable}\label{CarlesonSection} The aim of this section is to prove that if $E \subset \Hek$ is a closed upper Ahlfors-regular set, then balls with large horizontal width are quite rare. Namely, they satisfy the following Carleson packing condition.

\begin{proposition}\label{NMCarleson} Let $E \subset \Hek$ be a closed upper Ahlfors-regular set. Then,
\begin{displaymath}
\int_{0}^{R} \calHk (\{q \in E \cap B(p,R) : \wid_{B(q,s)}(E) > \epsilon\}) \,  \frac{ds}{s} \lesssim_{reg} \frac{R^{2k+1}}{\epsilon}
\end{displaymath}
for all $\epsilon > 0$, $p \in E$, and $R > 0$.
\end{proposition}

Proposition~\ref{NMCarleson} is a consequence of the following stronger result, which states that the map $(q,s) \mapsto \wid_{B(q,s)}(E)$ is $L^{1}$-integrable under mild assumptions. In the next proposition, an \emph{upper Ahlfors-regular measure} $\mu$ (with dimension $2k+1$) is a locally finite Borel measure for which there exists $C>0$ such that
\begin{displaymath}
\mu(B(p,r)) \leq C\, r^{2k+1}, \qquad p \in \Hek, \quad r>0.
\end{displaymath}

\begin{thm}\label{widthL1} Assume that $E \subset \Hek$ is a closed set and $\mu$ is an upper Ahlfors-regular measure. Then
\begin{displaymath}
\int_{0}^{\infty} \int_{\Hek} \wid_{B(q,s)}(E) \, d\mu(q) \, \frac{ds}{s} \leq C\calHk (E),
\end{displaymath}
where $C>0$ depends only on the upper Ahlfors-regularity constant of $\mu$.
\end{thm}

The only tool needed in the proof of Theorem~\ref{widthL1} is the following Crofton-formula type upper bound proven in~\cite[Lemma 5.3]{CFO2} in $\He^1$, whose proof easily extends to all higher dimensional Heisenberg groups. Namely, for every $\nu\in\Sk$, one has
\begin{displaymath}
\int^{\ast}_{\W_\nu} \card(E \cap \pi_{\W_\nu}^{-1}\{p\}) \, d\calHk (p) \lesssim_k \calHk (E), \qquad E \subset \Hek.
\end{displaymath}
Here $\int^{\ast}$ stands for the upper integral, see \cite[Chapter 1]{zbMATH01249699}. It is not hard to see that the integrand $p \mapsto \card(E \cap \pi_{\W_{\nu}}^{-1}\{p\})$ is a Borel function when $E \subset \Hek$ is closed. Now, recalling the definition~\eqref{e:def-meas-h-H2k+1} of the measure $\h$, it follows that
\begin{equation}\label{crofton2}
\int_\calL \card(E \cap \ell) \, d\h(\ell) = \int_{\Sk} \int_{\W_{\nu}} \card(E \cap \pi_{\W_\nu}^{-1}\{p\}) \, d\calHk (p) \, d\nu \lesssim_k \calHk (E)
\end{equation}
for $E \subset \Hek$ closed.

\begin{proof}[Proof of Theorem~\ref{widthL1}] With no loss of generality, assume that $\calHk (E) < +\infty$. Note that width is approximately monotone in the sense that
\begin{displaymath}
\wid_{B}(E) \lesssim_k \wid_{B'}(E)
\end{displaymath}
for all balls $B \subset B'$ with comparable radii (see~\eqref{e:def-totalwidth}). In particular, we have
\begin{equation}\label{form36}
\int_{0}^{\infty} \int_{\Hek} \wid_{B(q,s)}(E) \, d\mu(q) \, \frac{ds}{s} \lesssim_k \sum_{n \in \Z} \int_{\Hek} \wid_{B(q,2^{n})}(E) \, d\mu(q).
\end{equation}
Next, for $n \in \Z$ fixed, pick a maximal $2^{n}$-net $\{q^{n}_{0},q^{n}_{1},\ldots\} \subset \Hek$, and write $B_{i}^{n} := B(q^n_{i},2^{n + 1})$. Then
\begin{displaymath}
\wid_{B(q,2^{n})}(E) \lesssim_k \wid_{B_{i}^{n}}(E), \qquad q \in B(q_{i}^{n},2^{n}). \end{displaymath}
Since the balls $B(q_{i}^{n},2^{n})$, $i \in \N$, cover $\Hek$, we can further estimate as follows,
\begin{align*} \eqref{form36}
& \lesssim_k \sum_{n \in \Z} \sum_{i \in \N} \wid_{B_{i}^{n}}(E) \cdot \mu(B(q_{i},2^{n})) \lesssim_{reg} \sum_{n \in \Z} \sum_{i \in \N} \wid_{B_{i}^{n}}(E) \cdot 2^{(2k+1)n}\\
& \lesssim_{reg} \sum_{n \in \Z} \sum_{i \in \N} \int_\calL \frac{\diam(E \cap B_{i}^{n} \cap \ell)}{2^{n}} \, d\h(\ell) = \int_\calL \sum_{n \in \Z} \sum_{i \in \N} \frac{\diam(E \cap B_{i}^{n} \cap \ell)}{2^{n}} \, d\h(\ell), \end{align*}
where the implicit constants in the two last inequalities depend only on $k$ and the upper Ahlfors-regularity constant of $\mu$. Now, fix a horizontal line $\ell$ such that $\card(E \cap \ell) =: N < + \infty$. This is true for $\h$ almost every line, since $\calHk (E) < +\infty$. Let $a_{1},\ldots,a_{N}$ be an enumeration of the points in $E \cap \ell$, in increasing order for some fixed orientation of $\ell$. Write $I_{j} := [a_{j},a_{j + 1}]$ for $1 \leq j \leq N - 1$, and note that
\begin{displaymath}
\diam(E \cap B_{i}^{n} \cap \ell) = \sum_{I_{j} \subset B_{i}^{n}} \diam(I_{j}), \qquad n \in \Z, \quad i \in \N.
\end{displaymath}
Consequently,
\begin{displaymath}
\sum_{n \in \Z} \sum_{i \in \N} \frac{\diam(E \cap B_{i}^{n} \cap \ell)}{2^{n}} = \sum_{j = 1}^{N - 1} \sum_{\substack{n\in\Z,i\in\N, \\ I_{j} \subset B_{i}^{n}}} \frac{\diam(I_{j})}{2^{n}}
\end{displaymath}
For $j \in \{1,\ldots,N - 1\}$ fixed, the summation over $n$ can be restricted to those values with $2^{n+2} = \diam(B_{i}^{n})  \geq \diam(I_{j})$. Moreover, for every such $n$ fixed, the balls $B_{i}^{n}$, $i \in \N$, have bounded overlap, so $I_{j} \subset B^{n}_{i}$ can occur only for $\lesssim_k 1$ indices  $i \in \N$. These observations yield
\begin{displaymath}
\sum_{j = 1}^{N - 1} \sum_{\substack{n\in\Z,i\in\N, \\ I_{j} \subset B_{i}^{n}}} \frac{\diam(I_{j})}{2^{n}} \lesssim_k N - 1 \lesssim_k \card(E \cap \ell),
\end{displaymath}
so finally
\begin{displaymath} \int_{0}^{\infty} \int_{\Hek} \wid_{B(q,s)}(E) \, d\mu(q) \, \frac{ds}{s} \lesssim_{reg} \int_\calL \card(E \cap \ell) \, d\h(\ell) \lesssim_{reg} \calHk (E), \end{displaymath}
using \eqref{crofton2}, and the proof is complete. \end{proof}

Now we prove Proposition \ref{NMCarleson}.

\begin{proof}[Proof of Proposition~\ref{NMCarleson}] For $p\in E$ and $R>0$ fixed, we apply Theorem~\ref{widthL1} with $\mu := \calHk|_{E\cap B(p,R)}$. Note that whenever $q \in B(p,r)$ and $0 < r < R$, one has
\begin{displaymath} \wid_{B(q,s)}(E) = \wid_{B(q,s)}(E \cap B(p,2R)).
\end{displaymath}
This yields
\begin{align*}
\int_{0}^{R} & \calHk (\{q \in E \cap B(p,R) : \wid_{B(q,s)}(E) > \epsilon\}) \,  \frac{ds}{s} \\
&\leq \frac{1}{\epsilon} \int_0^{\infty} \int_{\Hek} \wid_{B(q,s)}(E \cap B(p,2R)) \, d\mu(q) \, \frac{ds}{s} \lesssim_{reg} \frac{\calHk (E\cap B(p,2R))}{\epsilon}.
\end{align*}
This completes the proof by the upper Ahlfors-regularity of $E$.
\end{proof}


\subsection{Bilateral weak geometric lemma for arbitrary hyperplanes} \label{sect:BWGL-arbitrary-hyperplanes}
In this section, we prove that a Semmes surface satisfies the bilateral weak geometric lemma for arbitrary hyperplanes. Roughly speaking, this means that it is bilaterally well-approximated by hyperplanes at most scales and locations, see Definition~\ref{def:bwgl-arbitrary-hyperplanes} and Proposition~\ref{prop:bwgl-arbitrary-hyperplanes}.

Given $E\subset\Hek$, $p\in E$, and $s>0$, we define the \textit{bilateral $\beta$-number $b\beta_E(p,s)$ for arbitrary hyperplanes} inside a given ball $B(p,s)$. We mimic here the Euclidean definition, see for instance~\cite[I.2.1]{MR1251061},  and set
\begin{equation*}
b\beta_E(p,s) := \inf_{P}\, \left\{ \sup_{q \in B(p,s) \cap E} \frac{\dist(q,P)}{s} + \sup_{q \in B(p,s) \cap P} \frac{\dist(q,E)}{s} \right\}
\end{equation*}
where the infimum runs over all hyperplanes $P\subset\Hek$. This definition takes into account the distance from points in $E$ to hyperplanes $P$, as well as the distance from points in $P$ to $E$.

\begin{definition}[BWGL for arbitrary hyperplanes] \label{def:bwgl-arbitrary-hyperplanes}
We say that a set $E\subset\Hek$ satisfies the \textit{bilateral weak geometric lemma} (or BWGL in short) \textit{for arbitrary hyperplanes} if
\begin{equation} \label{e:def-bwgl-arbitrary-hyperplanes}
\int_0^R \calHk (\{q\in E \cap B(p,R): b\beta_E(q,s) > \epsilon \}) \,  \frac{ds}{s} \lesssim_\epsilon R^{2k+1}
\end{equation}
for all $\epsilon>0$, $p\in E$, and $R>0$.
\end{definition}

The fact that Semmes surfaces satisfy BWGL for arbitrary hyperplanes comes as a rather immediate consequence of Corollary~\ref{cor:main} combined with Proposition~\ref{NMCarleson}.

\begin{proposition} [BWGL for arbitrary hyperplanes for Semmes surfaces] \label{prop:bwgl-arbitrary-hyperplanes}
Let $S \subset \Hek$ be a Semmes surface. Then $S$ satisfies BWGL for arbitrary hyperplanes. Moreover, the implicit constant on the right-hand side of~\eqref{e:def-bwgl-arbitrary-hyperplanes} can be chosen depending only on $\epsilon,k$ and on the upper Ahlfors-regularity and Condition~\textbf{B} constants of $S$.
\end{proposition}

\begin{proof}
Let $0<\epsilon < \overline \epsilon$ be fixed, where $\overline \epsilon$ is the dimensional constant given by Corollary~\ref{cor:main}, and let $0<\gamma<1$ be given by Corollary~\ref{cor:main}. Let $q\in S$ and $s>0$ be such that $b\beta_S(q,s) > \epsilon$. It follows from Corollary~\ref{cor:main} that $\wid_{B(q,\gamma^{-1}s)}(S) > (80 \gamma)^{2k+3}$. Then we get from Proposition~\ref{NMCarleson} that
\begin{equation*}
\begin{split}
\int_0^R \calHk (\{&q\in S \cap B(p,R): b\beta_S(q,s) > \epsilon \}) \,  \frac{ds}{s} \\
&\leq \int_0^{\gamma^{-1}R} \calHk (\{q\in S \cap B(p,\gamma^{-1} R): \wid_{B(q,\gamma^{-1}s)}(S) > (80 \gamma)^{2k+3} \}) \, \frac{ds}{s} \\
&\lesssim_\gamma (\gamma^{-1} R)^{2k+1}
\end{split}
\end{equation*}
which concludes the proof since $\gamma$ depends only on $\epsilon,k$ and the upper Ahlfors-regularity and Condition~\textbf{B} constants of $S$.
\end{proof}


\subsection{Bilateral weak geometric lemma for vertical hyperplanes}\label{sect:bwgl-vertical-hyperplanes} In this section, we upgrade Proposition~\ref{prop:bwgl-arbitrary-hyperplanes} from the previous section to a similar statement concerning approximation by \emph{vertical hyperplanes}, Proposition~\ref{prop:bwgl-vertical-hyperplanes}. We actually note that, more generally, the validity of BWGL for arbitrary hyperplanes, recall Definition~\ref{def:bwgl-arbitrary-hyperplanes}, for a closed Ahlfors-regular set in $\Hek$ is equivalent to the validity of BWGL for vertical hyperplanes, see Definition~\ref{def:bwgl-vertical-hyperplanes} and Theorem~\ref{thm:bwgl}. A proof of a result close to Theorem~\ref{thm:bwgl} is already implicitly contained in the reduction from~\cite[Proposition 75]{NY} to~\cite[Proposition 68]{NY}, so we will only give an outline of the proof here and refer to~\cite{NY} for part of the details.

Given $E\subset\Hek$, $p\in E$, and $s>0$, we define the \textit{bilateral $\beta$-number $\beta_{v,E} (p,s)$ for vertical hyperplanes} in a similar way than the bilateral $\beta$-number for arbitratry hyperplanes except that we restrict the infimum to run over vertical hyperplanes. Namely, we set
\begin{equation*}
b\beta_{v,E}(p,s) := \inf_{\W}\,  \left\{ \sup_{q \in B(p,s) \cap E} \frac{\dist(q,\W)}{s} + \sup_{q \in B(p,s) \cap \W} \frac{\dist(q,E)}{s} \right\}
\end{equation*}
where the infimum runs over all vertical hyperplanes $\W\subset\Hek$, that is, hyperplanes of the form $\W= p'\cdot \W_\nu$ for some $p'\in\Hek$ and $\nu\in\Sk$.

\begin{definition}[BWGL for vertical hyperplanes] \label{def:bwgl-vertical-hyperplanes}
We say that a set $E\subset\Hek$ satisfies the \textit{bilateral weak geometric lemma} (or BWGL in short) \textit{for vertical hyperplanes} if
\begin{equation} \label{e:def-bwgl-vertical-hyperplanes}
\int_0^R \calHk (\{q\in E \cap B(p,R): b\beta_{v,E}(q,s) > \epsilon \}) \,  \frac{ds}{s} \lesssim_\epsilon R^{2k+1}
\end{equation}
for all $\epsilon>0$, $p\in E$, and $R>0$.
\end{definition}

Note that, since $b\beta_E(q,s) \leq b\beta_{v,E}(q,s)$ for every $q\in E$ and $s>0$, it is immediate that BWGL for vertical hyperplanes implies BWGL for arbitrary hyperplanes. As already mentioned, it turns out that both versions of the bilateral weak geometric lemma are equivalent for Ahlfors-regular sets in $\Hek$.

\begin{thm} \label{thm:bwgl}
Let $E\subset\Hek$ be a closed Ahlfors-regular set. Then $E$ satisfies BWGL for vertical hyperplanes if and only if it satisfies BWGL for arbitrary hyperplanes.
\end{thm}

Going back to Semmes surfaces, we know from Proposition~\ref{prop:lowerAR} that they are lower Ahlfors-regular, and hence Ahlfors-regular, and from Proposition~\ref{prop:bwgl-arbitrary-hyperplanes} that they satisfy BWGL for abitrary hyperplanes. Hence the validity of BWGL for vertical hyperplanes follows from Theorem~\ref{thm:bwgl}.

\begin{proposition} [BWGL for vertical hyperplanes for Semmes surfaces] \label{prop:bwgl-vertical-hyperplanes}
Let $S \subset \Hek$ be a Semmes surface. Then $S$ satisfies BWGL for vertical hyperplanes. Moreover, the implicit constant in~\eqref{e:def-bwgl-vertical-hyperplanes} can be chosen depending only on $\epsilon,k$ and on the upper Ahlfors-regularity and Condition~\textbf{B} constants of $S$.
\end{proposition}

We give now an outline of the proof of Theorem~\ref{thm:bwgl} which is based on arguments developed in~\cite[Section~9.4]{NY}. Recall that if a hyperplane $P\subset\Hek$ is not vertical, then it is horizontal and there is a unique $p\in \Hek$ such that $P=p\cdot H$ (see the comment before Proposition~\ref{prop:small-non-mono-imply-approx-half-space}). In other words, a horizontal hyperplane $P$ coincides with the set of all horizontal lines passing through some point $p\in \Hek$ which we call the centre of $P$. As observed by Naor and Young, inside a ball far from its center, a horizontal hyperplane is close to some vertical hyperplane, see~\cite[Corollary~73]{NY}. Hence, if a set $E$ is close to a horizontal hyperplane $P$ in a ball far from the center of $P$, then $E$ is close to some vertical hyperplane in a slightly smaller ball.

Before proceeding further, and mainly for technical convenience, we recall the notion of \emph{David cubes} on Ahlfors-regular sets. A system of David cubes on a closed Ahlfors-regular set $E$ is a collection $\calD$ of subsets of $E$ with the following properties. We set $\mathbb{J} := \mathbb{Z}$ if $E$ is unbounded, and $\mathbb{J}:= \{j\in\Z: j\leq n\}$  where $n\in\Z$ is such that $2^{n} \leq \diam E <2^{n+1}$ if $E$ is bounded. First, $\calD = \cup_{j\in\mathbb{J}} \calD_j$ where, for each $j\in \mathbb{J}$, the letter $\calD_j$ denotes a family of disjoint subsets of $E$ such that $\calHk (E\setminus \cup_{Q \in \calD_j} Q ) = 0$. Next, for every $j\in\mathbb{J}$ and $Q \in \calD_j$, we have $\diam Q \lesssim_{reg} 2^{j}$, and there is a ball $B_Q = B(c_{Q},c2^{j})$, centred at some point $c_{Q} \in Q$ and called the centre of $Q$, such that $B_Q\cap E \subset Q$. The constant $c > 0$ depends only on the Ahlfors-regularity constants of $E$. Finally, for $i, j \in \mathbb{J}$ with $i\leq j$, if $Q\in \calD_i$ and $Q'\in\calD_j$, then either $Q\cap Q' = \emptyset$ or $Q \subset Q'$. For $j\in \mathbb{J}$ and $Q \in \calD_j$, we set $\ell(Q) := 2^{j}$. One should think of $\ell(Q)$ as a substitute for the ``side-length'' of the cube $Q$. The existence of systems of David cubes on Ahlfors-regular sets in Euclidean spaces has been proven in~\cite{MR1009120}. The construction has been extended to spaces of homogeneous type in~\cite{Christ}, to which we refer for more details.

To prove the non-trivial implication in Theorem~\ref{thm:bwgl}, we let $E\subset\Hek$ be a closed Ahlfors-regular set satisfying BWGL for arbitrary hyperplanes and $\calD$ be a system of David cubes on $E$. We fix a constant $C_0\geq 2$, depending only on the Ahlfors-regularity constants for $E$, such that
\begin{equation}\label{form40} 2\diam Q < C_0 \ell(Q), \qquad Q\in \calD. \end{equation}
For $Q \in \calD$, we set $b\beta_E(Q) := C_0 \, b\beta_E(c_Q,C_0\ell(Q))$, that is,
\begin{displaymath}
b\beta_E(Q) = \inf_{P} \left\{\sup_{p \in B(Q) \cap E} \frac{\dist(p,P)}{\ell(Q)} + \sup_{p \in B(Q) \cap P} \frac{\dist(p,E)}{\ell(Q)} \right\}
\end{displaymath}
where the infimum runs over all hyperplanes $P \subset \Hek$ and $B(Q):=B(c_{Q},C_{0}\ell(Q))$. Similarly, we set $b\beta_{v,E}(Q) := C_0 \, b\beta_{v,E} (c_Q,C_0\ell(Q))$, that is,
\begin{displaymath}
b\beta_{v,E}(Q) =  \inf_{\W} \left\{\sup_{p \in B(Q) \cap E} \frac{\dist(p,\W)}{\ell(Q)} + \sup_{p \in B(Q) \cap \W} \frac{\dist(p,E)}{\ell(Q)} \right\}
\end{displaymath}
where the infimum runs over all vertical hyperplanes $\W\subset\Hek$.

As a classical fact, it is easy to reformulate the BWGL conditions in terms of the numbers $b\beta_E(Q)$ and $b\beta_{v,E}(Q)$. Namely, a closed Ahlfors-regular set $E\subset\Hek$ satisfies BWGL for arbitrary, respectively vertical, hyperplanes if and only if, for every $\epsilon>0$ and $Q_0\in\calD$,
\begin{equation}  \label{e:bwgl-arbitrary-cubes}
\mathop{\sum_{Q \in \calD(Q_{0})}}_{b\beta_{E}(Q) > \epsilon} \ell(Q)^{2k+1} \lesssim_{\epsilon} \ell(Q_{0})^{2k+1},
\end{equation}
respectively
\begin{equation} \label{e:bwgl-vertical-cubes}
\mathop{\sum_{Q \in \calD(Q_{0})}}_{b\beta_{v,E}(Q) > \epsilon} \ell(Q)^{2k+1} \lesssim_{\epsilon} \ell(Q_{0})^{2k+1}.
\end{equation}
Here $\calD(Q_{0}):=\{Q\in\calD: Q\subset Q_0\}$.

\begin{remark} \label{rmk:bwgl} The precise value of the constant $C_{0}$ in the definition of $b\beta_E(Q)$ and $b\beta_{v,E}(Q)$ does not matter here. More precisely, the validity of~\eqref{e:bwgl-arbitrary-cubes}, respectively~\eqref{e:bwgl-vertical-cubes}, for a system of David cubes $\calD$ on $E$ and a choice of $C_0>0$ such that $\diam Q < C_0 \ell(Q)$ for $Q\in \calD$ implies the validity of BWGL for arbitrary, respectively vertical, hyperplanes, which is in turn equivalent to~\eqref{e:bwgl-arbitrary-cubes}, repectively~\eqref{e:bwgl-vertical-cubes}, for every system of David cubes $\calD$ on $E$ and every choice of $C_0>0$ such that $\diam Q < C_0 \ell(Q)$ for $Q\in \calD$. The implicit constants in~\eqref{e:bwgl-arbitrary-cubes} and~\eqref{e:bwgl-vertical-cubes} will then naturally depend on $C_0$. Our choice of $C_0$, recall \eqref{form40}, lies in this admissible range and is also suitable for our proof of Theorem~\ref{thm:bwgl}
(see in particular the proof of Lemma \ref{Lemma75NY}).
\end{remark}

Let $E$ and $\calD$ be as above. Fix $\epsilon > 0$ and a cube $Q_{0} \in \calD$, and let $\eta > 0$ be another parameter, to be fixed later small enough depending only on $\epsilon$, $k$, and on the Ahlfors-regularity constants for $E$. By Remark~\ref{rmk:bwgl}, the family of cubes
\begin{displaymath}
\calB_{0}(\eta) := \{Q \in \calD(Q_{0}) : b\beta_E(2Q) \geq \eta\}
\end{displaymath}
satisfies the Carleson packing estimate
\begin{displaymath}
\sum_{Q \in \calB_{0}(\eta)} \ell(Q)^{2k+1} \lesssim_{\eta} \ell(Q_{0})^{2k+1},
\end{displaymath}
where $b\beta_E(2Q):=2C_0 \, b\beta_E(c_Q,2C_0\ell(Q))$. We set $\calG_{0}(\eta) := \calD(Q_{0}) \setminus \calB_{0}(\eta)$ and we are going to define a subset $\calB_{1}(\epsilon)$ of $\calG_{0}(\eta)$ in such a way that, for all $Q \in \calG_{0}(\eta) \setminus \calB_{1}(\epsilon)$,
\begin{equation} \label{e:good-cubes}
b\beta_{v,E}(Q) \lesssim_{reg} \epsilon,
\end{equation}
and $\calB_{1}(\epsilon)$ will satisfy the Carleson packing estimate
\begin{equation}\label{e:bad-cubes}
\sum_{Q \in \calB_{1}(\epsilon)} \ell(Q)^{3} \lesssim_{\epsilon} \ell(Q_{0})^{3},
\end{equation}
provided $\eta$ is chosen small enough, depending only on $\epsilon$, $k$, and on the Ahlfors-regularity constants of $E$. This will complete the proof of Theorem~\ref{thm:bwgl}.

As in~\cite{NY}, given $p\in\Hek$ and a hyperplane $P\subset \Hek$, we set
\begin{displaymath}
\alpha_{p}(P) :=
\begin{cases}
d_{\mathrm{Euc}}(\pi(p),\pi(q)) & \text{if } P = q \cdot H \text{ is horizontal}, \\
+\infty & \text{if } P \text{ is vertical.}
\end{cases}
\end{displaymath}
Here $\pi \colon \Hek \to \Rk$ stands for the projection $\pi(v,t) := v$ which is a $1$-Lipschitz mapping between $(\Hek,d)$ and $(\Rk,d_{\mathrm{Euc}})$. By definition, $\alpha_p(P)$ is a number in $[0,+\infty]$ and $\alpha_p(P) = +\infty$ if and only if the hyperplane $P$ is vertical.  On the other hand, when $\alpha_p(P)$ is finite, then the hyperplane $P$ is horizontal and one has
\begin{equation*}
\angle (\mathbb{T}, p^{-1}\cdot P ) = \arctan (2/\alpha_p (P))
\end{equation*}
where $\angle(\mathbb{T}, p^{-1}\cdot P )$ denotes the minimum Euclidean angle between the vertical axis $\mathbb{T}:=\{(0,t) \in \Hek : t\in \R\}$ and a line in $p^{-1}\cdot P$ that passes through the point where $\mathbb{T}$ intersects $p^{-1}\cdot P$. Here we interpret "$\arctan (2/0)$" as "$\pi / 2$". This can be deduced from the definition of $\alpha_p(P)$ together with elementary computations. It follows that large values of $\alpha_p (P)$ imply that, inside every ball centred at $p$, the hyperplane $P$ is quantitatively close to some vertical hyperplane, see~\cite[Corollary~73]{NY}. In particular, we get from~\cite[Corollary~73]{NY} combined with~\cite[Lemma~52(3)]{NY} that, for every $Q\in\mathcal{D}$, every $q\in Q$, and every hyperplane $P$, there is a vertical hyperplane $V$ such that
\begin{equation}  \label{e:P-V}
\sup_{p \in 2B(Q) \cap P} \dist(p,V) + \sup_{p \in 2B(Q) \cap V} \dist(p,P) \lesssim_{reg}  \ell(Q)^2 / \alpha_q(P)
\end{equation}
where $2B(Q):=B(c_Q,2C_0\ell(Q))$.

We go back now to $\calG_{0}(\eta)$. We associate to every $Q \in \calG_{0}(\eta)$ a hyperplane $P_{Q}$ with
\begin{equation} \label{e:bwgl}
\sup_{p \in 2B(Q) \cap E} \dist(p,P_{Q}) + \sup_{p \in 2B(Q) \cap P_Q} \dist(p,E) \leq \eta \ell(Q).
\end{equation}
We set
\begin{equation*}
\alpha(Q) := \inf_{p \in Q} \frac{\alpha_{p}(P_Q)}{\ell(Q)}
\end{equation*}
and
\begin{displaymath}
\calB_{1}\left(\epsilon
\right) := \{Q \in \calG_{0}(\eta) : \alpha(Q) < 1/\epsilon\}.
\end{displaymath}

It follows from~\eqref{e:P-V} and~\eqref{e:bwgl} that every $Q \in \calG_{0}(\eta) \setminus \calB_{1}(\epsilon)$ satisfies~\eqref{e:good-cubes}, provided $\eta$ is chosen small enough, depending only on $\epsilon$, $k$, and on the Ahlfors-regularity constants of $E$. To conclude the proof, it thus remains to prove that the Carleson packing estimate~\eqref{e:bad-cubes} holds, provided $\eta$ is chosen even smaller if necessary, and depending only on $\epsilon$, $k$, and on the Ahlfors-regularity constants of $E$.

The proof of~\eqref{e:bad-cubes} runs essentially in the same way as the end of the proof of~\cite[Proposition~75]{NY}, but we rephrase a few details. They are virtually the same as in~\cite{NY}, but our set-up is quite different from that in~\cite[Proposition~75]{NY} (see Section~\ref{subsect:NY-vs-our-setting} for further comments). Also, our notion of ``close to a hyperplane'' is somewhat different from the notion of ``close to a half-space'' employed in~\cite{NY}.

Following~\cite{NY}, we say that a cube $R \in \calG_{0}(\eta)$ is a \emph{good descendant} of $Q \in \calG_{0}(\eta)$, if all the cubes $Q' \in \calD$ with $R \subset Q' \subset Q$ lie in $\calG_{0}(\eta)$. For $Q \in \calG_{0}(\eta)$, we denote by $G(Q)$ the good descendants of $Q$, and we write
\begin{displaymath}
G(Q,\epsilon):= \{R \in G(Q) : \alpha(R) < 1/\epsilon\}.
\end{displaymath}
The next lemma is an analogue of \cite[Lemma~76]{NY}.
\begin{lemma}\label{Lemma75NY} Let $0<\epsilon <1$. If $\eta > 0$  is sufficiently small, depending only on $\epsilon$, $k$, and on the Ahlfors-regularity constants for $E$, then $G(Q,\epsilon)$ is a tree for all $Q \in \calG_{0}(\eta)$. More precisely, whenever $R \in G(Q,\epsilon)$ and $Q' \in \calD$ with $R \subset Q' \subset Q$, then $Q' \in G(Q,\epsilon)$. Further, the cubes in any fixed tree $G(Q,\epsilon)$, $Q \in \calG_{0}(\eta)$, satisfy a Carleson packing condition,
\begin{equation*}
\sum_{R \in G(Q,\epsilon)} \ell(R)^{2k+1} \lesssim_{\epsilon} \ell(Q)^{2k+1}.
\end{equation*}
\end{lemma}

\begin{proof} To prove the tree property, fix $Q \in \calG_{0}(\eta)$ and $R \in G(Q,\epsilon)$. It suffices to show that the parent of $R$, say $Q'$, lies in $G(Q,\epsilon)$. First, since $R \in G(Q)$ by definition, also $Q' \in G(Q)$. So, it remains to show that $\alpha(Q') < 1/\epsilon$. Since $\alpha(R) < 1/\epsilon$, we can find a point $p \in R$ with
\begin{equation*}
\alpha_{p}(P_{R}) < \ell(R)/\epsilon.
\end{equation*}
Now, we would like to argue that
\begin{equation*}
\alpha_{p}(P_{Q'}) < \ell(Q')/\epsilon
\end{equation*}
if $\eta > 0$ is sufficiently small. Since $p \in R \subset 2B(R)\cap E$, we have by the choice of $P_R$ that $\dist(p,P_R) \leq \eta \ell (R)$. By our choice of the constant $C_0$, we also have $B(p,\ell (R)) \subset B(R) \subset 2B(R) \subset 2B(Q')$. By the choice of $P_R$ and $P_{Q'}$, it follows that
\begin{equation*}
\sup_{q \in  B(p,\ell (R)) \cap P_R} \dist (q,P_{Q'}) + \sup_{q\in B(p,\ell (R)) \cap P_{Q'}} \dist (q,P_R) \leq 3\eta \ell(R),
\end{equation*}
at least provided $\eta$ is small enough, depending only on $C_0$, and hence only on the Ahlfors-regularity constants for $E$. This forces in particular the Euclidean angle between $P_{R}$ and $P_{Q'}$ to be small and, taking into account that $\alpha_p (P_R) < \ell (R)/ \epsilon$, it follows from the proof of~\cite[Lemma~74]{NY} that $\alpha_p(P_{Q'}) < 2 \ell (R) / \epsilon = \ell(Q') / \epsilon$, provided $\eta$ is chosen small enough, depending only on $\epsilon$, $k$, and on the Ahlfors-regularity constants for $E$. This proves that $\alpha(Q') < 1/\epsilon$, hence $Q'\in G(Q,\epsilon)$ and the tree property has been established.

The rest of the proof is now verbatim the same as in~\cite[Lemma~76]{NY}, noting that, as we just did in the previous argument, the $U$-local distance between the half-spaces $P_Q^+$ and $P_R^{+}$ considered in~\cite{NY} (see~\cite[Definition~51]{NY}) has to be replaced by quantities of the form
$$\sup_{p\in U \cap P_Q} \dist(p,P_R) + \sup_{p\in U\cap P_R} \dist(p,P_Q)$$
in our setting. We will thus not repeat further the details.
\end{proof}

With Lemma~\ref{Lemma75NY} in hand, the remainder of the proof of Theorem~\ref{thm:bwgl}, that is, the proof of~\eqref{e:bad-cubes}, is precisely the same as the end of the proof of~\cite[Proposition~75]{NY}, so we also omit the details. This concludes the proof of Theorem~\ref{thm:bwgl}.


\section{Semmes surfaces have big pieces of intrinsic Lipschitz graphs} \label{sect:BPiLG} We conclude in this section the proof of Theorem~\ref{main}. 
We know from Proposition~\ref{prop:bwgl-vertical-hyperplanes} that Semmes surfaces satisfy the bilateral weak geometric lemma for vertical hyperplanes. As an immediate consequence, we get the validity of the weak geometric lemma for vertical hyperplanes of which we recall now the definition.

Given $E\subset\Hek$, $p\in E$, and $s>0$, we define the \textit{$\beta$-number $\beta_{v,E} (p,s)$ for vertical hyperplanes} in a similar way than the bilateral $\beta$-number except that we only take into account the distance from points in $E$ to vertical hyperplanes inside a given ball $B(p,s)$, that is, we set
\begin{equation*}
\beta_{v,E}(p,s) := \inf_{\W}\,  \sup_{q \in B(p,s) \cap E} \frac{\dist(q,\W)}{s}
\end{equation*}
where the infimum runs over all vertical hyperplanes $\W\subset\Hek$.

\begin{definition}[WGL for vertical hyperplanes] \label{def:wgl-vertical-hyperplanes}
We say that a set $E\subset\Hek$ satisfies the \textit{weak geometric lemma} (or WGL in short) \textit{for vertical hyperplanes} if
\begin{equation} \label{e:def-wgl-vertical-hyperplanes}
\int_0^R \calHk (\{q\in E \cap B(p,R): \beta_{v,E}(q,s) > \epsilon \}) \,  \frac{ds}{s} \lesssim_\epsilon R^{2k+1}
\end{equation}
for all $\epsilon>0$, $p\in E$, and $R>0$.
\end{definition}

Since $\beta_{v,E}(q,s) \leq b\beta_{v,E}(q,s)$ for every $q\in E$ and $s>0$, it is immediate that BWGL for vertical hyperplanes implies WGL for vertical hyperplanes and the following proposition is then a straightforward consequence of Proposition~\ref{prop:bwgl-vertical-hyperplanes}.

\begin{proposition} [WGL for vertical hyperplanes for Semmes surfaces] \label{prop:wgl-vertical-hyperplanes}
Let $S \subset \Hek$ be a Semmes surface. Then $S$ satisfies WGL for vertical hyperplanes. Moreover, the implicit multiplicative constant in the right-hand side of~\eqref{e:def-wgl-vertical-hyperplanes} can be chosen depending only on $\epsilon$, $k$, and on the upper Ahlfors-regularity and Condition~\textbf{B} constants for $S$.
\end{proposition}

To conclude the proof of Theorem~\ref{main}, recall that we also know from Propositions~\ref{prop:lowerAR} and~\ref{BVPProp} that Semmes surfaces are lower Ahlfors-regular, and hence Ahlfors-regular, and have big vertical projections. We then use the fact that a closed Ahlfors-regular set with BVP and satisfying WGL for vertical hyperplanes in $\Hek$ has big pieces of intrinsic Lipschitz graphs. In $\He^1$, this result is due to the first two authors with Chousionis~\cite[Theorem~3.7]{CFO}. The extension of this result to all Heisenberg groups is given in Section~\ref{sect:WGL+BVP-BPiLG-H2k+1}, see Theorem~\ref{thm:BVP+WGL-BPiLG}.


\section{BVP + WGL for vertical hyperplanes imply BPiLG} \label{sect:WGL+BVP-BPiLG-H2k+1}

It has been proven in~\cite[Theorem~3.7]{CFO} that a closed Ahlfors-regular set with big vertical projections (BVP, Definition~\ref{def:BVP}) and satisfying the weak geometric lemma (WGL) for vertical hyperplanes (Definition~\ref{def:wgl-vertical-hyperplanes}) in $\He^1$ has big pieces of intrinsic Lipschitz graphs (BPiLG, Definition~\ref{def:BPiLG}). We extend in this section the proof given in~\cite{CFO} to all Heisenberg groups.

\begin{thm} \label{thm:BVP+WGL-BPiLG}
Let $k\geq 1$ and $E\subset \Hek$ be a closed Ahlfors-regular set. Assume that $E$ has BVP and satisfies WGL for vertical hyperplanes. Then $E$ has BPiLG.
\end{thm}

The proof given in~\cite{CFO} for $k=1$ relies on two steps,~\cite[Lemma~3.8]{CFO} and~\cite[Theorem~3.9]{CFO}, together with a concluding argument. The proof of~\cite[Theorem~3.9]{CFO} and the concluding argument can be verbatim extended to $\Hek$, so we omit the details. The higher dimensional version of~\cite[Lemma~3.8]{CFO} is given in Lemma~\ref{lem:Lemma3.9CFO-H2k+1}.

\begin{lemma} \label{lem:Lemma3.9CFO-H2k+1}
Let $E\subset\Hek$ be a closed Ahlfors-regular set. For every $c>0$ and $M>1$, there are $\epsilon>0$ and $\gamma >0$, depending only on $c$, $M$, $k$, and the Ahlfors-regularity constants for $E$, such that the following holds. Let $\calD$ be a system of David cubes on $E$ and $Q\in \calD$. Assume that $\calHk (\pi_{\W_\nu} (Q))  \geq c \calHk (Q)$ for some $\nu\in\Sk$ and $\beta_{v,E}(Q) \leq \epsilon$. Then, for every $p\in Q$ and $q\in B(Q) \cap E$ such that $M^{-1} l(Q) \leq d(p,q) \leq M l(Q)$, one has $q \not\in p \cdot C_\gamma (\nu)$.
\end{lemma}

Here $\beta_{v,E}(Q):= C_0 \, \beta_{v,E}(c_Q,C_0\ell(Q))$, that is,
\begin{displaymath}
\beta_{v,E}(Q)= \inf_{\W}\,  \sup_{q \in B(Q) \cap E} \frac{\dist(q,\W)}{\ell(Q)}
\end{displaymath}
where the infimum runs over all vertical hyperplanes $\W$. Recall from Section~\ref{sect:bwgl-vertical-hyperplanes} that $B(Q):=B(c_Q, C_0\ell(Q))$.

\begin{proof} We assume that the conclusion of the lemma fails in the following sense. We consider a closed Ahlfors-regular set $E \subset \Hek$ with a system of David cubes $\calD$ and $Q\in \calD$ such that
\begin{itemize}
\item $0\in Q$ and $\ell(Q)=1$,
\item $\calHk (\pi_{\W_\nu} (Q))  \geq c\calH^{2k + 1}(Q)$ for some $\nu\in\Sk$,
\item $\beta_{v,E}(Q) \leq \epsilon$, and
\item there exists $q\in B(Q) \cap E$ such that $M^{-1} \leq d(0,q) \leq M$ and $q \in C_\gamma (\nu)$.
\end{itemize}
We will show that these conditions lead to a contradiction. It is not difficult to reduce the general case to this one, see the beginning of the proof of~\cite[Lemma~3.8]{CFO} for more details.

Here is a sketch of the proof. If $\gamma$ is very small, then $q \in C_{\gamma}(\nu)$ implies that $q \approx \pi_{\mathbb{L}_{\nu}}(q)$. Since also $q \in Q$ and $d(0,q) \sim 1$, it follows that a small neighbourhood of the well-approximating vertical hyperplane $\W_{\nu'}$ -- guaranteed by $\beta_{v,E}(Q) \leq \epsilon$ -- must contain $\mathbb{L}_{\nu} \cap B(0,C_{0})$. But $\pi_{\W_{\nu}}(\mathbb{L}_{\nu}) = \{0\}$, so $\pi_{\W_{\nu}}(\W_{\nu'} \cap B(0,C_{0}))$ is contained in a small neighbourhood of a codimension-$1$ subspace of $\W_{\nu}$. But $Q$ is well-approximated by $\W_{\nu'} \cap B(0,C_{0})$, so also $\pi_{\W_{\nu}}(Q)$ is contained in a small neighbourhood of a codimension-$1$ subspace of $\W_{\nu}$. This contradicts $\calH^{2k + 1}(\pi_{\W_{\nu}}(Q)) \geq c\calH^{2k + 1}(Q) \sim 1$.

We turn to the details. Let $q'\in \Hek$ and $\nu'\in\Sk$ be such that
\begin{equation} \label{e:beta-W'}
\sup_{p' \in B(Q) \cap E} \dist(p',q' \cdot \W_{\nu'}) \leq 2\epsilon.
\end{equation}

We first show that there is $u\in\Sk$ such that $\R u \subset \pi(\W_{\nu'})$ and the unoriented Euclidean angle $\angle (u,\nu)$ between $u$ and $\nu$ satisfies $\angle (u,\nu) \lesssim_{M} (\gamma + \epsilon)$, provided $\epsilon$ and $\gamma$ are chosen small enough depending only on $M$. On the one hand, we have $q=(v,t) \in C_\gamma (\nu)$, that is,
$$\|\pi_{\W_\nu}(q)\| \leq \gamma \|\pi_{\LL_\nu}(q)\|$$ with $\pi_{\W_\nu}(q) =  (v-\langle v,\nu \rangle \nu, t-\omega(v,\langle v,\nu \rangle \nu)/2)$ and $\pi_{\LL_\nu}(q) = (\langle v,\nu \rangle \nu,0)$, see~\eqref{e:pwnu} and~\eqref{e:plnu}. Then, taking into account that, by assumption, $d(0,q) \sim_{M} 1$, computations similar to those used to obtain~\cite[(3.8)]{CFO} give $|v| \sim_{M} 1$, provided $\gamma$ is chosen small enough. On the other hand, by~\cite[Remark~3.4]{CFO}, which extends to the $\Hek$ setting, we know that~\eqref{e:beta-W'} implies that, for all $p'\in B(Q) \cap E$,
\begin{equation} \label{e:beta-W'-bis}
\dist (p',\W_{\nu'}) \leq 4 \epsilon.
\end{equation}
In particular, since $q\in B(Q) \cap E$, there is $(v',t') \in \W_{\nu'}$ such that $d(q,(v',t'))\leq 4 \epsilon$. Hence $|v-v'| \leq d(q,(v',t')) \leq 4 \epsilon$. Since $|v| \sim_{M} 1$, it follows that $|v'| \sim_{M} 1$ provided $\epsilon$ is chosen small enough depending only on $M$. Then we get $$|v'-\langle v',\nu \rangle \nu| \leq |v-\langle v,\nu \rangle \nu| + |(v'-\langle v',\nu \rangle \nu) - (v-\langle v,\nu \rangle \nu)| \lesssim_{M} (\gamma + \epsilon) |v'|. $$
Hence $u:= v'/|v'| \in \Sk$ is such that $\R u \subset \pi(\W_{\nu'})$ and $\angle (u,\nu) \lesssim_{M} (\gamma + \epsilon)$ as required.

We set $W_\nu:= \pi(\W_\nu) = \nu^\perp$ and $W_{\nu'}:=\pi(\W_{\nu'})$. By the previous argument, we get that, for $\epsilon$ and $\gamma$ chosen small enough depending only on $M$, we have $W_\nu + W_{\nu'} = \R^{2k}$. Thus $\dim(W_\nu \cap W_{\nu'})= 2k-2$ and $W_{\nu'} = \R u \oplus (W_\nu \cap W_{\nu'})$. We set $\mathbb{V}:= (W_\nu \cap W_{\nu'}) \times \R$ and we next show that there is a constant $\Lambda>0$ depending only on the Ahlfors-regularity constants for $E$, such that, for every $\tau\in (0,1)$, one can choose $\epsilon$ and $\gamma$ small enough depending only on $\tau$, $M$, and on the Ahlfors-regularity constants for $E$, so that
\begin{equation}  \label{e:proj-subset-tau-neighbourhood}
\pi_{\W_\nu} (Q) \subset (\mathbb{V} \cap B(0,\Lambda))_\tau:=\{p'\in\W_\nu :\, \dist(p',\mathbb{V} \cap B(0,\Lambda)) \leq \tau\}.
\end{equation}
Indeed, let $p'\in Q \subset B(Q) \cap E$. By~\eqref{e:beta-W'-bis} there is $(w',s') \in \W_{\nu'}$ such that $d(p',(w',s')) \leq 4 \epsilon$. We have $w'=a u + \hat{w}$ for some $a\in\R$ and $\hat{w} \in W_\nu \cap W_{\nu'}$. We set $w'':=a\nu + \hat{w}$ and $p'':=(w'',s')$. Then
\begin{equation*}
\begin{split}
\pi_{\W_\nu}(p'') &= (w'' - \langle w'',\nu \rangle\nu, s'-\omega(w'',\langle w'',\nu \rangle\nu)/2) \\
&= (\hat{w},s'-\omega(\hat{w},a\nu)/2).
\end{split}
\end{equation*}
Hence $\pi_{\W_\nu}(p'') \in \mathbb{V}$. Furthermore, we have $$\|\pi_{\W_\nu}(p'')\| \leq 2 d(0,p'')$$ with $d(0,p'') \leq d(0,(w',s')) + d((w',s'),p'')$. Since $p'\in Q$ and $d(p',(w',s')) \leq 4 \epsilon$, we have $d(0,(w',s')) \lesssim_{reg} 1$. Next, this implies that $|a\langle u,\nu\rangle| \leq |w'| \leq d(0,(w',s')) \lesssim_{reg} 1$. Since $\angle (u,\nu) \lesssim_{M} (\gamma + \epsilon)$, one can choose $\epsilon$ and $\gamma$ small enough depending only on $M$ so that $1/2 \leq \langle u,\nu \rangle \leq 1$ and $|u-\nu|\leq (\gamma + \epsilon)$. Then we get $|a|  \lesssim_{reg} 1$ and it follows that $|w''-w'| = |a(\nu - u)|  \lesssim_{reg} (\gamma + \epsilon)$. We also have $|\hat{w}| = |au-w'| \leq |a| + |w'| \lesssim_{reg} 1$ and
\begin{equation*}
\begin{split}
\omega(w',w'') &= \omega(au+\hat{w},a\nu + \hat{w})\\
&= a^2 \omega(u,\nu) + a \omega(u-\nu,\hat{w})\\
&= a^2 \omega(u-\nu,\nu) + a  \omega(u-\nu,\hat{w}),
\end{split}
\end{equation*}
hence, $$|\omega(w',w'')| \leq a^2 |u-\nu| \, |\nu| + |a| \, |u-\nu| \, |\hat{w}| \lesssim_{reg} |u-\nu| \lesssim_{reg}  \gamma + \epsilon.$$
Since $(w',s')^{-1} \cdot p'' = (w''-w',-\omega(w',w'')/2)$, we get $$d((w',s'),p'')\lesssim_{reg} (\gamma+\epsilon)^{1/2}.$$ It follows that $\|\pi_{\W_\nu}(p'')\| \lesssim_{reg} 1$. Hence, there is $\Lambda>0$, which depends only on the Ahlfors-regularity constant for $E$, such that $\pi_{\W_\nu}(p'') \in B(0,\Lambda)$ provided $\epsilon$ and $\gamma$ are chosen small enough depending only on $M$. Remembering that vertical projections are locally $1/2$-H\"older continuous, we then get
\begin{equation*}
\begin{split}
\dist(\pi_{\W_\nu}(p'), \mathbb{V} \cap B(0,\Lambda)) &\leq d(\pi_{\W_\nu}(p'),\pi_{\W_\nu}(p'')) \\
&\lesssim_{reg} d(p',p'')^{1/2}\\
&\lesssim_{reg} (d(p',(w',s')) + d((w',s'),p''))^{1/2} \lesssim_{reg} (\gamma + \epsilon)^{1/4},
\end{split}
\end{equation*}
which concludes the proof of~\eqref{e:proj-subset-tau-neighbourhood}.

To conclude the proof of the lemma, we first note that $\calHk(\mathbb{V} \cap B(0,\Lambda)) = 0$. This can be seen from the fact that $\mathbb{V}$ is a $(2k-1)$-dimensional linear subspace of $\W_\nu$ and $\calHk|_{\W_\nu}$ coincides, up to a multiplicative constant, with the $2k$-dimensional Lebesgue measure when identifying $\W_\nu$ with $\mathbb{R}^{2k-1}\times \R$. Hence $$\lim_{\tau\rightarrow 0} \calHk((\mathbb{V} \cap B(0,\Lambda))_\tau) = 0.$$ It follows that one can choose $\tau >0$ small enough, depending only on $c$ and on the Ahlfors-regularity constant for $E$, so that, for $\epsilon>0$ and $\gamma>0$ chosen small enough accordingly, and hence, depending only on $c$, $M$, and on the Ahlfors-regularity constant for $E$, one has $\calHk(\pi_{\W_\nu}(Q)) \leq \calHk((\mathbb{V} \cap B(0,\Lambda))_\tau) < c \calHk(Q)$ (recall that $\ell(Q)= 1$ and hence $\calHk(Q)\sim_{reg} 1$) and this gives a contradiction.
\end{proof}


\section{Comments} \label{sect:final}


\subsection{Regular triples and corona decompositions} \label{subsect:NY-vs-our-setting}

Mimicking the terminology introduced in~\cite{NY}, given a measurable set $E\subset \Hek$, we say that $(E,E^c, \partial E)$ is \textit{regular} if there is $c>0$ such that
\begin{equation} \label{e:int-upper-reg-2k+2}
\calH^{2k+2} (B(p,r)\cap E) \geq c\, r^{2k+2} , \qquad  p \in E, \quad 0 < r < \diam \partial E,
\end{equation}
\begin{equation} \label{e:ext-upper-reg-2k+2}
\calH^{2k+2} (B(p,r)\setminus E) \geq c\, r^{2k+2} , \qquad  p \in E^c, \quad 0 < r < \diam \partial E,
\end{equation}
and $\partial E$ is Ahlfors-regular (with dimension $2k+1$, see Definition~\ref{def:Ahlfors-regular}).

Local versions of regular triples, for which~\eqref{e:upper-regular},~\eqref{e:lower-regular},~\eqref{e:int-upper-reg-2k+2} and~\eqref{e:ext-upper-reg-2k+2} are required to hold for $r \in (0,r_0]$ for some $r_0>0$, have been considered in~\cite{NY}. We could also have considered local versions for the definition of a Semmes surface $S$, where~\eqref{e:upper-regular} and~\eqref{corkscrew} are required to hold for $r \in (0,r_0]$ for some $r_0>0$. Our arguments in the present paper also apply in this case to give big pieces of intrinsic Lipschitz graphs inside balls centred on $S$ with radius $r \in (0,r_0]$. For the sake of simplicity, we discuss here the link between (global) regular triples and (global) Semmes surfaces, noting that these links can be easily extended for the local versions.

The link between regular triples and open sets satisfying the corkscrew condition, Definition~\ref{def:corkscrew-cdt}, is given in the next proposition.

\begin{proposition} \label{prop:NY-vs-Semmes-surfaces} Let $S\subset \Hek$. The following are equivalent.
\begin{itemize}
\item[\textit{(i)}] There is a measurable set $E\subset \Hek$ such that $(E,E^c, \partial E)$ is regular with $S=\partial E$.
\item[\textit{(ii)}] The set $S$ is a closed upper Ahlfors-regular set and there is an open set $\Omega\subset \Hek$ satisfying the corkscrew condition such that $S=\partial \Omega$.
\end{itemize}
\end{proposition}

\begin{proof}
If $\textit{(ii)}$ holds, then $E:=\Omega$ satisfies~\eqref{e:int-upper-reg-2k+2} and~\eqref{e:ext-upper-reg-2k+2} as a rather immediate consequence of the corkscrew condition. Moreover, when~$\textit{(ii)}$ holds, then $S=\partial \Omega$ is a Semmes surface and we know from Proposition~\ref{prop:lowerAR} that $S$ is lower Ahlfors-regular and hence Ahlfors-regular. It follows that $(\Omega,\Omega^c,\partial \Omega)$ is regular and hence $\textit{(i)}$ holds.

To prove the converse implication, we let $E\subset \Hek$ be measurable such that $(E,E^c, \partial E)$ is regular with $S=\partial E$. We denote by $\Omega$ the interior of $E$ and we prove that $S=\partial \Omega$ and $\Omega$ satisfies the corkscrew condition. Since $\partial \Omega \subset \partial E = S$, to prove the first claim, we only need to check that $S\subset \partial \Omega$. Let $p\in S$. Then~\eqref{e:int-upper-reg-2k+2} and~\eqref{e:ext-upper-reg-2k+2} imply that for every $0<r<\diam S$,  we have
\begin{equation*}
\min\{\calH^{2k+2} (B(p,r)\cap E), \calH^{2k+2} (B(p,r)\setminus E)\} \gtrsim_{reg} r^{2k+2},
\end{equation*}
where $\gtrsim_{reg}$ means here that the implicit multiplicative constant depends on $k$ and on the regularity constant $c$ coming from~\eqref{e:int-upper-reg-2k+2} and~\eqref{e:ext-upper-reg-2k+2}. Since $\partial E$ is Ahlfors-regular, we have $\calH^{2k+2} (\partial E) = 0$, hence $$\calH^{2k+2} ( E \bigtriangleup \Omega) = \calH^{2k+2} ( E \bigtriangleup \overline \Omega)=0.$$ It follows that
\begin{multline*}
\min \{\calH^{2k+2} (B(p,r)\cap \Omega), \calH^{2k+2}  (B(p,r)\setminus \overline \Omega)\}\\
 =\min\{\calH^{2k+2} (B(p,r)\cap E), \calH^{2k+2} (B(p,r)\setminus E)\} \gtrsim_{reg} r^{2k+2}
\end{multline*}
for every $0<r<\diam S$. Hence $p\in\partial\Omega$ and $S\subset \partial \Omega$ as claimed.

To prove that $\Omega$ satisfies the corkscrew condition, we let $p\in \partial \Omega$ and $0<r < \diam \partial \Omega$. It follows from the previous argument that
\begin{equation} \label{e:lower-Omega}
\min\{\calH^{2k+2} (B(p,r/4)\cap \Omega), \calH^{2k+2} (B(p,r/4)\setminus \overline \Omega)\} \gtrsim_{reg} r^{2k+2}.
\end{equation}
Let $0<t<1/4$ and $\mathcal{A}$ be a maximal family of points in $S\cap B(p,r/2)$ at mutual distance $>tr$. Since the balls $B(q,tr/2)$, $q\in\mathcal{A}$, are disjoint and contained in $B(p,r)$, we get from the Ahlfors-regularity of $S=\partial\Omega$ that
\begin{equation*}
\begin{split}
(tr)^{2k+1} \card \mathcal{A} &\lesssim_{reg} \sum_{q\in \mathcal{A}} \calHk (S\cap B(q,tr/2))\\
&\lesssim_{reg} \calHk (S\cap B(p,r)) \lesssim_{reg} r^{2k+1},
\end{split}
\end{equation*}
and hence $\card \mathcal{A} \lesssim_{reg} t^{-(2k+1)}$. Next, since $$\{q' \in B(p,r/4):\, \dist(q',S) \leq tr \} \subset \bigcup_{q\in \mathcal{A}} B(q,2tr),$$ we get
\begin{equation*}
\calH^{2k+2} (\{q' \in B(p,r/4):\, \dist(q',S) \leq tr \}) \lesssim_{reg} t r^{2k+2}.
\end{equation*}
Choosing $t$ small enough, depending only on the regularity constants for $(E,E^c,\partial E)$, this last estimate together with~\eqref{e:lower-Omega} implies the existence of $q_1 \in B(p,r/4)\cap \Omega$ and $q_2 \in B(p,r/4)\setminus \overline \Omega$ with $\dist(q_i, S) > tr$, $i=1,2$. Hence $\Omega$ satisfies the corkscrew condition and this concludes the proof of the proposition.
\end{proof}

It follows
from Proposition~\ref{prop:NY-vs-Semmes-surfaces} that $\partial E$ is a Semmes surface whenever $(E,E^c, \partial E)$ is regular. However, as already mentioned in the introduction, there are Semmes surfaces that do not arise as the boundary of some open set satisfying the corkscrew condition. A simple example is given by the union of the unit sphere $\partial B(0,1)$ with the intersection of the unit ball $B(0,1)$ with a hyperplane through the origin. Also, a connected component of $S^c$ where $S$ is a Semmes surface may not satisfy the corkscrew condition. Examples in the Euclidean setting can for instance be found in~\cite{MR3738187}. Hence the setting in the present paper includes boundaries of sets $E$ for which $(E,E^c, \partial E)$ is regular, or equivalently, boundaries of open sets satisfying the corkscrew condition with upper Ahlfors-regular boundary, but is slightly more general. 

It is proven in~\cite[Section~9]{NY} that if $(E,E^c,\partial E)$ is regular then the pair $(E,\partial E)$ admits a ``corona decomposition''. We will not enter the details of the definition of such a decomposition here, and refer to~\cite[Definition~53]{NY}, and~\cite[I.3]{MR1251061} for its Euclidean analogue.
In Euclidean spaces,
having big pieces of Lipschitz graphs (BPLG) for a closed Ahlfors-regular set implies the existence of a corona decomposition.
The latter is one of several characterisations of uniform rectifiability in $\mathbb{R}^n$ and it is equivalent to the validity of the bilateral weak geometric lemma. In the Euclidean setting, it is also known that BPLG is stronger than uniform rectifiability, as there are examples of uniformly rectifiable sets without big projections, and hence, without BPLG. In the Heisenberg setting, the possible links, or differences, between the analogues of these various notions are not well understood at the time. Understanding them better would be one further step in the development of the theory of ``uniform rectifiability'' in Heisenberg groups.


\subsection{Semmes surfaces in Euclidean space}\label{subsect:euclidean-setting} In this section, we discuss the analogues of our arguments in the Euclidean setting. We recall that a Semmes surface in $\R^n$ is a closed upper Ahlfors-regular set with dimension $n-1$ that satisfies Condition~\textbf{B}. The definitions of upper Ahlfors-regularity with dimension $n-1$ and of Condition~\textbf{B} in $\R^n$ are analogous to the versions we stated for the Heisenberg group in Definition~\ref{def:Ahlfors-regular} and Definition~\ref{def:conditionB}, except that all metric concepts are now, as well as in the rest of this section, defined with respect to the Euclidean distance. We refer to~\cite[Definition~2.2]{MR1251061} or~\eqref{e:def-bwgl-eucl} below for the definition of the bilateral weak geometric lemma in $\R^{n}$.

We first recall that a set $S$ satisfying Condition~\textbf{B} in $\R^n$ is automatically lower Ahlfors-regular, and hence Ahlfors-regular, with dimension $n-1$ and even has big projections. This is a rather immediate consequence of the definitions together with the fact that orthogonal projections onto subspaces are Euclidean Lipschitz maps. Then, the fact that Semmes surfaces in $\R^{n}$ satisfy BWGL has been proven by David and Semmes through an intermediate ``local symmetry condition'' \cite[Theorem~1.20]{MR1132876} together with~\cite[Proposition~5.5]{DS1}. We record here that the method of the current paper works in $\R^{n}$, too, and hence gives a new and direct proof of BWGL for Semmes surfaces in Euclidean spaces.

The definitions of width, non-convexity, and non-monotonicity with respect to a line in a ball are verbatim the same in $\R^{n}$ as in $\Hek$, except that we consider all possible lines, that is, all one-dimensional affine subspaces of $\R^n$. The conclusion of Lemma~\ref{lemma1} holds for every line in $\R^n$
since every line $\ell$ has now the crucial feature that for all line segments $I\subset\ell$, one has $\mathrm{diam}(I) = \mathcal{H}^1(I)$.

It is well-known that there exists a unique, up to a multiplicative constant, non-trivial isometry-invariant measure on the set $\calL_{\R^n}$ of all lines in $\R^{n}$. We denote such a measure by $\eta$. Then, in perfect analogy with~\eqref{NMDef} and~\eqref{e:def-totalwidth}, we define
\begin{displaymath} \NM_{B(x,r)}(A) := \frac{1}{r^{n}} \int_{\calL_{\R^n}} \NM_{B(x,r)}(A,\ell) \, d\eta(\ell) \end{displaymath}
and
\begin{displaymath} \wid_{B(x,r)}(E) := \frac{1}{r^{n}} \int_{\calL_{\R^n}} \wid_{B(x,r)}(E,\ell) \, d\eta(\ell) \end{displaymath}
for measurable sets $A \subset \R^{n}$ and closed sets $E \subset \R^{n}$. The conclusion of Proposition~\ref{form1} can then obviously be rephrased in the Euclidean setting using these latter definitions.

We recall now that the classification of monotone sets in $\R^n$ is much easier than its analogue in $\Hek$ and can for instance be found in~\cite[Lemma~64]{NY}, see also~\cite[Lemma~4.2]{CK}. Then the proof of Lemma~\ref{lem:non-monotonicity-vs-monotone-sets} and Proposition~\ref{prop:small-non-mono-imply-approx-half-space} can be easily translated to the Euclidean setting to give the following

\begin{proposition}
For every $C>0$ and $\delta>0$, there exists $0< \gamma < 1$ such that the following holds. If $F \subset \R^n $ is measurable with $C$-upper Ahlfors-regular boundary with dimension $n-1$, $p\in \R^n$, $r>0$, and $\NM_{B(p,r)}(F) \leq  \gamma^{n + 1}$, then there is a half-space $P^-\subset \R^n$ such that
\begin{equation*}  \frac{\calH^{n}([F \bigtriangleup P^{-}] \cap B(x,\gamma r))}{\calH^{n}(B(x,\gamma r))} \leq \delta. \end{equation*}
\end{proposition}

Next, the proof of Proposition~\ref{monotonicityToBetas} can be translated to the Euclidean setting without any modifications, except for the trivial ones, to give the following

\begin{proposition} \label{prop:eucl-monotonicityToBetas}
There is a dimensional constant $\overline \epsilon >0$ such that the following holds for every $0 <\epsilon < \overline \epsilon$. Assume that $S \subset \R^n$ is a closed set satisfying Condition~\textbf{B}. There exists $\delta > 0$, depending only on $\epsilon,n$ and the Condition~\textbf{B} constant for $S$, such that if $p \in S$ and $0<r <\diam S$ are such that for every component $\Omega$ of $S^c$, there exists a half-space $P_{\Omega}^{-} \subset \R^n$ with
\begin{equation*}
\frac{\calH^{n}([\Omega \bigtriangleup P_{\Omega}^{-}] \cap B(p,r))}{\calH^{n}(B(p,r))} \leq \delta,
\end{equation*}
 then, there exists a hyperplane $P \subset \R^n$ such that $\dist(q,P) \leq \epsilon r$ for all $q \in S \cap B(p,r/80)$ and $\dist(q,S) \leq \epsilon r$ for all $q\in P\cap B(p,r/80)$.
\end{proposition}

As a consequence, one gets the following Euclidean version of Corollary~\ref{cor:main}.

\begin{cor} \label{cor:eucl-small-width-to-beta}
There is a dimensional constant $\overline \epsilon >0$ such that the following holds for every $0 <\epsilon < \overline \epsilon$. Let $S\subset \R^n$ be a Semmes surface. There is $0<\gamma<1$, depending only on $\epsilon,n$ and on the upper Ahlfors-regularity and Condition~\textbf{B} constants for $S$, such that the following holds. If $p\in S$, $0<r<\diam S$, and $\wid_{B(p,r)}(S) \leq (80\gamma)^{n+1}$, then there is a hyperplane $P\subset \R^n$ such that
\begin{equation*}
\sup_{q\in S \cap B(p, \gamma r)} \dist(q,P) + \sup_{q\in P \cap B(p, \gamma r)} \dist(q,S) \leq \epsilon \gamma r.
\end{equation*}
\end{cor}

Finally, the proof of Theorem~\ref{widthL1} and Proposition~\ref{NMCarleson} can be rephrased to give the following statements in $\R^{n}$.

\begin{thm}\label{widthL1Rn} Assume that $E \subset \R^{n}$ is a closed set and $\mu$ is an upper Ahlfors-regular measure with dimension $n-1$ in $\R^{n}$. Then
\begin{displaymath} \int_{0}^{\infty} \int_{\R^{n}} \wid_{B(x,s)}(E) \, d\mu(x) \, \frac{ds}{s} \lesssim_{reg} \calH^{n - 1}(E). \end{displaymath}
\end{thm}

\begin{proposition} \label{prop:carleson-width-eucl}
Let $E \subset \R^n$ be a closed upper Ahlfors-regular set with dimension $n-1$. Then,
\begin{displaymath} \int_{0}^{R} \calH^{n-1}(\{q \in E \cap B(p,R) : \wid_{B(q,s)}(E) > \epsilon\}) \,  \frac{ds}{s} \lesssim_{reg} \frac{R^{n-1}}{\epsilon} \end{displaymath}
for all $\epsilon > 0$, $p \in E$, and $R > 0$.
\end{proposition}

The only noticeable difference in the proof of Theorem~\ref{widthL1Rn} is that~\eqref{crofton2} should be replaced by
\begin{displaymath} \int_{\calL_{\R^n}} \card(E \cap \ell) \, d\eta(\ell) \lesssim_n \calH^{n - 1}(E), \end{displaymath}
which holds by the same argument as for~\eqref{crofton2}.

We recall now the definition of the bilateral weak geometric lemma for sets with codimension one in the Euclidean setting \cite[Definition~2.2]{MR1251061}. Given $E\subset\R^n$, $p\in E$, and $s>0$, we define
\begin{equation*}
b\beta_E(p,s) := \inf_{P}\, \left\{ \sup_{q \in B(p,s) \cap E} \frac{\dist(q,P)}{s} + \sup_{q \in B(p,s) \cap P} \frac{\dist(q,E)}{s} \right\},
\end{equation*}
where the infimum runs over all hyperplanes $P\subset\R^n$. We say that a set $E\subset\R^n$ satisfies the \textit{bilateral weak geometric lemma} (or BWGL in short) if
\begin{equation} \label{e:def-bwgl-eucl}
\int_0^R \calH^{n-1} (\{q\in E \cap B(p,R): b\beta_E(q,s) > \epsilon \}) \,  \frac{ds}{s} \lesssim_\epsilon R^{n-1}
\end{equation}
for all $\epsilon>0$, $p\in E$, and $R>0$.

The validity of BWGL for Semmes surfaces in $\R^n$ can now be obtained as an easy consequence of Corollary~\ref{cor:eucl-small-width-to-beta} and Proposition~\ref{prop:carleson-width-eucl}, following the arguments given for the proof of Proposition~\ref{prop:bwgl-arbitrary-hyperplanes}.

To conclude, we go back to the analogue of Theorem~\ref{main} in the Euclidean setting. As already said, as an easy consequence of Condition~\textbf{B}, Semmes surfaces in $\R^n$ have big projections and are lower Ahlfors-regular. Since BWGL obviously imply WGL (see for instance~\cite[Definition~1.16]{MR1132876} for the Euclidean version of the weak geometric lemma), one can implement the arguments described above to recover that Semmes surfaces in $\R^n$ have big pieces of Lipschitz graphs (BPLG), the last step of the proof being given by~\cite[Theorem~1.14]{MR1132876} which states that a closed Ahlfors-regular set with big projections and satisfying WGL in $\R^n$ has BPLG.

\bibliographystyle{plain}
\bibliography{references}

\def\cprime{$'$}
\begin{thebibliography}{10}

\bibitem{MR3738187}
Jonas Azzam.
\newblock Tangents, rectifiability, and corkscrew domains.
\newblock {\em Publ. Mat.}, 62(1):161--176, 2018.

\bibitem{MR3626548}
Jonas Azzam, Steve Hofmann, Jos\'e~Maria Martell, Kaj Nystr\"om, and Tatiana
  Toro.
\newblock A new characterization of chord-arc domains.
\newblock {\em J. Eur. Math. Soc. (JEMS)}, 19(4):967--981, 2017.

\bibitem{Badger}
Matthew Badger.
\newblock Null sets of harmonic measure on {NTA} domains: {L}ipschitz
  approximation revisited.
\newblock {\em Math. Z.}, 270(1-2):241--262, 2012.

\bibitem{MR2312336}
L.~Capogna, D.~Danielli, S.~D. Pauls, and J.~T. Tyson.
\newblock {\em An introduction to the {H}eisenberg group and the
  sub-{R}iemannian isoperimetric problem}, volume 259 of {\em Progress in
  Mathematics}.
\newblock Birkh\"auser Verlag, Basel, 2007.

\bibitem{CG}
Luca Capogna and Nicola Garofalo.
\newblock Boundary behavior of nonnegative solutions of subelliptic equations
  in {NTA} domains for {C}arnot-{C}arath\'eodory metrics.
\newblock {\em J. Fourier Anal. Appl.}, 4(4-5):403--432, 1998.

\bibitem{MR2250055}
Luca Capogna and Nicola Garofalo.
\newblock Ahlfors type estimates for perimeter measures in
  {C}arnot-{C}arath\'{e}odory spaces.
\newblock {\em J. Geom. Anal.}, 16(3):455--497, 2006.

\bibitem{MR1847513}
Luca Capogna, Nicola Garofalo, and Duy-Minh Nhieu.
\newblock Examples of uniform and {NTA} domains in {C}arnot groups.
\newblock In {\em Proceedings on {A}nalysis and {G}eometry ({R}ussian)
  ({N}ovosibirsk {A}kademgorodok, 1999)}, pages 103--121. Izdat. Ross. Akad.
  Nauk Sib. Otd. Inst. Mat., Novosibirsk, 2000.

\bibitem{MR1323792}
Luca Capogna and Puqi Tang.
\newblock Uniform domains and quasiconformal mappings on the {H}eisenberg
  group.
\newblock {\em Manuscripta Math.}, 86(3):267--281, 1995.

\bibitem{CK}
Jeff Cheeger and Bruce Kleiner.
\newblock Metric differentiation, monotonicity and maps to {$L^1$}.
\newblock {\em Invent. Math.}, 182(2):335--370, 2010.

\bibitem{CKN}
Jeff Cheeger, Bruce Kleiner, and Assaf Naor.
\newblock Compression bounds for {L}ipschitz maps from the {H}eisenberg group
  to {$L_1$}.
\newblock {\em Acta Math.}, 207(2):291--373, 2011.

\bibitem{CFO}
V.~{Chousionis}, K.~{F{\"a}ssler}, and T.~{Orponen}.
\newblock {Intrinsic Lipschitz graphs and vertical $\beta$-numbers in the
  Heisenberg group}.
\newblock {\em ArXiv:1606.07703v3, accepted for publication in Amer. Journal of
  Math.}

\bibitem{CFO2}
V.~{Chousionis}, K.~{F{\"a}ssler}, and T.~{Orponen}.
\newblock {Boundedness of singular integrals on $C^{1,\alpha}$ intrinsic graphs
  in the Heisenberg group}.
\newblock {\em ArXiv e-prints}, August 2017.

\bibitem{Christ}
M.~Christ.
\newblock A {$T(b)$} theorem with remarks on analytic capacity and the {C}auchy
  integral.
\newblock {\em Colloq. Math.}, 60/61(2):601--628, 1990.

\bibitem{DJ}
G.~David and D.~Jerison.
\newblock Lipschitz approximation to hypersurfaces, harmonic measure, and
  singular integrals.
\newblock {\em Indiana Univ. Math. J.}, 39(3):831--845, 1990.

\bibitem{DS1}
G.~David and S.~Semmes.
\newblock Singular integrals and rectifiable sets in {${\bf R}^n$}: Au-del\`a
  des graphes lipschitziens.
\newblock {\em Ast\'erisque}, (193):152, 1991.

\bibitem{MR1251061}
G.~David and S.~Semmes.
\newblock {\em Analysis of and on uniformly rectifiable sets}, volume~38 of
  {\em Mathematical Surveys and Monographs}.
\newblock American Mathematical Society, Providence, RI, 1993.

\bibitem{MR1009120}
Guy David.
\newblock Morceaux de graphes lipschitziens et int\'egrales singuli\`eres sur
  une surface.
\newblock {\em Rev. Mat. Iberoamericana}, 4(1):73--114, 1988.

\bibitem{MR1132876}
Guy David and Stephen Semmes.
\newblock Quantitative rectifiability and {L}ipschitz mappings.
\newblock {\em Trans. Amer. Math. Soc.}, 337(2):855--889, 1993.

\bibitem{MR3511465}
B.~Franchi and R.~Serapioni.
\newblock Intrinsic {L}ipschitz graphs within {C}arnot groups.
\newblock {\em J. Geom. Anal.}, 26(3):1946--1994, 2016.

\bibitem{FSSC}
B.~Franchi, R.~Serapioni, and F.~Serra~Cassano.
\newblock Rectifiability and perimeter in the {H}eisenberg group.
\newblock {\em Math. Ann.}, 321(3):479--531, 2001.

\bibitem{FSS}
B.~Franchi, R.~Serapioni, and F.~Serra~Cassano.
\newblock Intrinsic {L}ipschitz graphs in {H}eisenberg groups.
\newblock {\em J. Nonlinear Convex Anal.}, 7(3):423--441, 2006.

\bibitem{MR2313532}
Bruno Franchi, Raul Serapioni, and Francesco Serra~Cassano.
\newblock Regular submanifolds, graphs and area formula in {H}eisenberg groups.
\newblock {\em Adv. Math.}, 211(1):152--203, 2007.

\bibitem{MR1404326}
Nicola Garofalo and Duy-Minh Nhieu.
\newblock Isoperimetric and {S}obolev inequalities for
  {C}arnot-{C}arath\'{e}odory spaces and the existence of minimal surfaces.
\newblock {\em Comm. Pure Appl. Math.}, 49(10):1081--1144, 1996.

\bibitem{JK}
David~S. Jerison and Carlos~E. Kenig.
\newblock Boundary behavior of harmonic functions in nontangentially accessible
  domains.
\newblock {\em Adv. in Math.}, 46(1):80--147, 1982.

\bibitem{MR3165282}
Riikka Korte and Panu Lahti.
\newblock Relative isoperimetric inequalities and sufficient conditions for
  finite perimeter on metric spaces.
\newblock {\em Ann. Inst. H. Poincar\'{e} Anal. Non Lin\'{e}aire},
  31(1):129--154, 2014.

\bibitem{LR}
G.~P. Leonardi and S.~Rigot.
\newblock Isoperimetric sets on {C}arnot groups.
\newblock {\em Houston J. Math.}, 29(3):609--637, 2003.

\bibitem{zbMATH01249699}
P.~{Mattila}.
\newblock {\em {Geometry of sets and measures in Euclidean spaces. Fractals and
  rectifiability. 1st paperback ed.}}
\newblock Cambridge: Cambridge University Press, 1st paperback ed. edition,
  1999.

\bibitem{MR2165404}
Francescopaolo Montefalcone.
\newblock Some relations among volume, intrinsic perimeter and one-dimensional
  restrictions of {BV} functions in {C}arnot groups.
\newblock {\em Ann. Sc. Norm. Super. Pisa Cl. Sci. (5)}, 4(1):79--128, 2005.

\bibitem{MM}
Roberto Monti and Daniele Morbidelli.
\newblock Regular domains in homogeneous groups.
\newblock {\em Trans. Amer. Math. Soc.}, 357(8):2975--3011, 2005.

\bibitem{Naor:2017:IGG:3055399.3055413}
A.~Naor and R.~Young.
\newblock The {I}ntegrality {G}ap of the {G}oemans-{L}inial {SDP} {R}elaxation
  for {S}parsest {C}ut is at {L}east a {C}onstant {M}ultiple of $\sqrt{\log
  n}$.
\newblock In {\em Proceedings of the 49th Annual ACM SIGACT Symposium on Theory
  of Computing}, STOC 2017, pages 564--575, New York, NY, USA, 2017. ACM.

\bibitem{NY}
Assaf Naor and Robert Young.
\newblock Vertical perimeter versus horizontal perimeter.
\newblock {\em Ann. of Math. (2)}, 188(1):171--279, 2018.

\bibitem{Semmes}
Stephen~W. Semmes.
\newblock A criterion for the boundedness of singular integrals on
  hypersurfaces.
\newblock {\em Trans. Amer. Math. Soc.}, 311(2):501--513, 1989.

\bibitem{MR3587666}
F.~Serra~Cassano.
\newblock Some topics of geometric measure theory in {C}arnot groups.
\newblock In {\em Geometry, analysis and dynamics on sub-{R}iemannian
  manifolds. {V}ol. 1}, EMS Ser. Lect. Math., pages 1--121. Eur. Math. Soc.,
  Z\"urich, 2016.

\end{thebibliography}

\end{document}